\documentclass[a4paper, 10pt, twoside]{amsart}
\usepackage[top=3cm, bottom=3cm, left=3.5cm, right=3.5cm]{geometry}
\usepackage[pdftex]{graphicx}

\usepackage[utf8]{inputenc}
\usepackage[T1]{fontenc}
\usepackage[english]{babel}
\usepackage[noadjust]{cite}
\usepackage{amsthm,amssymb,epsfig,mathtools}
\usepackage{mathrsfs}  % for \mathscr 
\usepackage{bbm}        % bold math
\usepackage{esint}     % for /fint
\usepackage{dsfont}
\usepackage{color}
\usepackage{enumerate}
\usepackage{url}
\usepackage{algorithm2e}

\usepackage{calc}
\usepackage[inline]{enumitem}
\usepackage[normalem]{ulem}
\usepackage{url}

\usepackage[bookmarks, bookmarksnumbered, pdfstartview={XYZ null null 1.00}]{hyperref}

\usepackage{tikz}
\usetikzlibrary{shapes, arrows, arrows.meta, decorations.markings}

\usepackage{todonotes}
\usepackage[ulem=normalem,draft]{changes}
\usepackage{cancel}

\makeindex

\numberwithin{equation}{section}

\newcommand{\theoname}{Theorem}
\newcommand{\lemmname}{Lemma}
\newcommand{\coroname}{Corollary}
\newcommand{\propname}{Proposition}
\newcommand{\definame}{Definition}

\newcommand{\remkname}{Remark}
\newcommand{\explname}{Example}

\theoremstyle{plain}
\newtheorem{theorem}{\theoname}[section]
\newtheorem{lemma}[theorem]{\lemmname}

\theoremstyle{definition}

\newtheorem{remark}[theorem]{\remkname}
\newtheorem{example}[theorem]{\explname}

\newlist{hypothesis}{enumerate}{1}
\setlist[hypothesis]{label={\textup{(H\arabic*)}}, ref={(H\arabic*)}, leftmargin=*, widest*=10}

\newenvironment{block}%
  {\list{}{\leftmargin=.5in\rightmargin=.5in}  \item[]  }%
  {\endlist}

\def\dd{{\rm d}}
\newcommand{\eqdef}{\ensuremath{\stackrel{\mbox{\upshape\tiny def.}}{=}}}

\newcommand{\norm}[1]{\left\lVert#1\right\rVert}
\newcommand{\inner}[1]{\left\langle#1\right\rangle}

\def\1B{{\bf  1}}

\def\det{\mathop{\rm det}}

\def\dist{{\rm dist}}

\def\intt{\mathop{\rm int}}

\def\inf{\mathop{\rm inf}}
\def\sup{\mathop{\rm sup}}

\def\min{\mathop{\rm min}}
\def\max{\mathop{\rm max}}

\def\argmin{\mathop{\rm argmin}}

\newcommand{\mres}{\mathbin{\vrule height 1.6ex depth 0pt width
		0.13ex\vrule height 0.13ex depth 0pt width 1.3ex}}

\newcommand\proj{\mathop{\rm proj}}

\newcommand{\nc}{\newcommand}
\nc{\grad}{{\mbox{grad}\,}}
\nc{\R}{{\mathbb R}}
\nc{\Rn}{{\mathbb R}^n}
\nc{\N}{{\mathbb N}}
\nc{\Z}{{\mathbb Z}}
\nc{\K}{{\cal K}}
\nc{\kpo}{(\K^d_{2,gp})_o}
\nc{\krep}{\K^d_{3,gp}}
\nc{\BP}{\mathbb{P}}
\nc{\BQ}{\mathbb{Q}}
\nc{\BE}{\mathbb{E}}
\nc{\cH}{\mathcal{H}}
\nc{\cL}{\mathcal{L}}
\nc{\cE}{\mathcal{E}}
\nc{\BS}{\mathbb{S}}
\nc{\LL}{L^\circ}
\nc{\bB}{B}
\nc{\sph}{\mathbb{S}^{n-1}}
\nc{\Q}{\mathbb{Q}}
\nc{\cK}{\mathcal{K}}
\nc{\vaps}{\varepsilon}
\nc{\cV}{\mathcal{V}}
\nc{\cP}{\mathcal{P}}
\nc{\cR}{\mathcal{R}}
\nc{\cS}{\mathcal{S}}
\nc{\cB}{\mathcal{B}}
\nc{\fB}{\mathfrak{B}}
\nc{\cC}{\mathcal{C}}
\nc{\ver}{{\rm{vert\, }}}
\nc{\iF}{\,\raisebox{0.0pt}{$\square$}\,}
\nc{\ind}{\mathbf{1}}	
\nc{\fed}{\llcorner}
\DeclareMathOperator{\BM}{\mathrm{BM}}
\DeclareMathOperator{\dc}{\mathrm{d}_\mathrm{C}}
\DeclareMathOperator{\dH}{\mathrm{d}_\mathrm{H}}

\title[Quantitative Stability for Minkowski's problem]{Quantitative Stability for Minkowski's problem}

\author{Károly Böröczky}
\address{Alfréd Rényi Institute of Mathematics, Hungarian Academy of Sciences, Realtanoda u.
13-15, H-1053, Budapest, Hungary}
\email{boroczky.karoly.j@renyi.hu}
\author{João Miguel Machado}
\address{Lagrange Mathematical and Computational Center\\
103 rue de Grenelle\\
Paris, 75007, France}
\email{joao-miguel.machado@ceremade.dauphine.fr}
\author{João P. G. Ramos}
\address{Instituto Nacional de Matem\'atica Pura e Aplicada \\ Estrada Dona Castorina 110, Horto - Rio de Janeiro, RJ - Brazil}
\email{joao.ramos@impa.br}

\begin{document}

\begin{abstract}
We derive quantitative stability results for Minkowski bodies, as well as their counterparts, the $L_p$-Minkowski bodies in the range $1 \le p \neq n$. We prove that, for every pair of probability measures $\mu,\nu$ satisfying a quantitative form of the classical dispersion assumptions yielding existence of such bodies, we have a control of the form 
\[
    \inf_{x\in \mathbb{R}^n}\mathrm{d_H}(E_\mu, x + E_\nu) \le C \mathrm{d_C}(\mu,\nu)^{\frac{1}{n-1}}, \quad 
    \alpha(E_\mu, E_\nu)^2 \le C \mathrm{d_C}(\mu,\nu)^{1 + \frac{1}{n-1}},
\]
where $\mathrm{d_H}$ denotes the Hausdorff distance, $\alpha$ denotes the Fraenkel asymmetry and $\mathrm{d_C}$ is the dual-convex distance of probability measures on the sphere. Our arguments are based on a variational problem whose optimizers are Minkowski bodies, for which we can obtain strong-concavity properties with the quantitative Brunn-Minkowski and isoperimetric inequalities. While the exponent in the Hausdorff distance is sharp, the exponent in the Fraenkel asymmetry is optimal in dimension $2$. 
\bigskip

\noindent\textbf{Keywords.} $L_p$ Minkowski problem, quantitative stability, Wulff shape

\medskip

\noindent\textbf{2020 Mathematics Subject Classification.} 52A21, 52A40, 49Q10

\end{abstract}

\maketitle
\tableofcontents

\section{Overview on Minkowski problems}\label{sec.introduction}

Given a non-trivial finite Borel measure $\mu$ on the unit sphere $\mathbb{S}^{n-1}$ of the Euclidean space $\mathbb{R}^n$, Minkowski asked under what conditions  one can  find a convex body whose normal vectors are distributed according to $\mu$. This is now known as Minkowski's problem and has been studied extensively since the beginning of the 20th century. More precisely, let $E \subset \mathbb{R}^n$ be a convex body and let $\nu_E : \partial E \mapsto \mathbb{S}^{n-1}$ denote its Gauss map, when $\partial E$ is a $(n-1)$-dimensional $\mathscr{C}^1$ manifold $\nu_E(x)$ coincides with the outwards unit normal of $\partial E$ at the point $x$, then the surface measure of $E$ is given by 
\begin{equation}\label{eq.surface_measure_def}
    S_E \eqdef {(\nu_E)}_\sharp \mathcal{H}^{n-1}\mres \partial E, 
\end{equation}
where $\mathcal{H}^{n-1}$ denotes the $(n-1)$-dimensional Hausdorff measure of $\mathbb{R}^n$. Since $S_{\lambda\,E}=\lambda^{n-1}S_E$ holds for any $\lambda>0$, we may assume that our Borel measure $\mu$ is a probability distribution; namely, $\mu \in \mathscr{P}(\mathbb{S}^{n-1})$.
The \textit{Minkowski problem} can then be phrased as:
\begin{block}
    Given $\mu \in \mathscr{P}(\mathbb{S}^{n-1})$, is there a convex body $E$ such that $S_E = \mu $? 
\end{block}
Extending Minkowski's work~\cite{minkowski1903volumen,minkowski1911theorie}, this question was answered by Alexandrov~\cite{aleksandrov1938theory,aleksandrov1996ad} positively if, and only if, the barycenter in $\mathbb{R}^n$ of the measure $\mu$ on $\mathbb{S}^{n-1}$ is the origin and $\mu$ is not supported on any hyperplane. He showed that these objects are unique up to translations, and we let $E_\mu$ be the unique convex body centered at the origin such that $S_{E_\mu} = \mu$. These are called \textit{Minkowski bodies}.

Whenever $\dd\mu=f\dd(\mathcal{H}^{n-1}\mres\mathbb{S}^{n-1})$, the support functions associated with Minkowski bodies solve a corresponding Monge-Amp\`ere equation on $\mathbb{S}^{n-1}$ given
\begin{equation}
\label{MinkowskiMongeAmpere}
    \det(\nabla_{\mathbb{S}^{n-1}}^2 h+h I)=f, 
\end{equation}
where $\nabla_{\mathbb{S}^{n-1}} h$ is the spherical gradient and
$\nabla_{\mathbb{S}^{n-1}}^2 h$ is the spherical Hessian of $h$ with respect to a moving frame. Regularity of the solution of \eqref{MinkowskiMongeAmpere} was intensively investigated by 
Nirenberg~\cite{nirenberg1953weyl}, Cheng, Yau \cite{cheng1976regularity} and Pogorelov \cite{pogorelov1964monge}, and the final solution was only provided by Caffarelli \cite{caffarelli1990interior,caffarelli1990localization} in 1990   (see Figalli \cite{figalli2017monge} and Trudinger, Wang \cite{trudinger2008monge} for a detailed discussion).

In the present work, we address the question of stability of Minkowski bodies. That is, can we estimate a suitable distance between two Minkowski bodies $E_\mu, E_\nu$ with a suitable distance between $\mu,\nu$? It is well known that Minkowski bodies can be obtained as solutions to an appropriated variational problem, which will be discussed further on, so that the question of stability can be seen as stability of minimizers with respect to a parameter of the energy functional.

The stability issue for this problem was already addressed by Hug-Schneider~\cite{hug2002stability} following the works of Diskant~\cite{diskant1972bounds} on the stability of the isoperimetric inequality. They show a bound of the form 
\[
    \inf_{x_0 \in \mathbb{R}^n} 
    \dH(E_\mu, x_0 + E_\nu)^n 
    \le C \dc(\mu,\nu), 
\]
where the constant $C$ depends on the inner and circumradii of $E_\mu$ and $E_\nu$, $\dH$ is the Hausdorff distance between convex bodies, and $\dc$ denotes the \textit{convex-dual distance} defined for two measures $\mu,\nu$ over $\mathbb{S}^{n-1}$ as 
\begin{equation}\label{eq.convex_dual_distance}
    \dc(\mu,\nu) 
    \eqdef 
    \sup_{K \subseteq B_1(0)} 
    \left|
        \int_{\mathbb{S}^{n-1}} h_K \dd (\mu-\nu)
    \right|,
\end{equation}
where the supremum is taken over all convex bodies included in the unit ball, and $h_K$ corresponds to the support function of $K$, defined in~\eqref{eq.support_function}. Their work was then leveraged by Abdallah and Mérigot~\cite{abdallah2015reconstruction} in the task of reconstructing convex bodies from random measurements of their normals.

In the present work, we prove stability results assuming the form
\begin{equation}\label{eq.BM_stability_prototype}
    \begin{aligned}
        \inf_{x_0 \in \mathbb{R}^n} \dH(E_\mu, x_0 + E_\nu) 
        &\le C {\dc(\mu,\nu)}^{\frac{1}{n-1}}, \\ 
        {\alpha(E_\mu, E_\nu)}^2 
        &\leq C {\dc(\mu,\nu)}^{1 + \frac{1}{n-1}},
    \end{aligned}
\end{equation}
for some constant $C$. Here $\alpha$ denotes the \textit{Fraenkel asymmetry}, used for stability statements by~\cite{fusco2008sharp} in the case of the isoperimetric inequality and by \cite{figalli2010stability} in the case of the more general anisotropic isoperimetric inequality, defined  as 
\begin{equation}\label{eq.fraenkel_asymmetry}
    \alpha(E,F) 
    \eqdef 
    \inf
    \left\{
        \frac{|E \Delta (x_0 + rF)|}{|E|}: 
        \substack{
            \displaystyle
            x_0 \in \mathbb{R}^n,\\ 
            \displaystyle
            r^n|F| = |E|,\; r>0
        }
    \right\}.
\end{equation}
It holds that $\alpha(E,F)=\alpha(F,E)$, and $\alpha(E,F)=0$ if and only if $E$ and $F$ are homothetic, so that $\alpha$ becomes a distance on the equivalence class of homotheties. Furthermore, in examples~\ref{exemple.sharp_hausdorff} and~\ref{exemple.sharp_fraenkel} we show that the exponent for the Hausdorff distance estimate is sharp, and that the exponent for the Fraenkel asymmetry is sharp in dimension $2$, showing therefore a Lipschitz behavior. 

Let us briefly discuss the choice of distances in~\eqref{eq.BM_stability_prototype}. Since the support functions $h_K$ are $1$-Lipschitz for $K \subset B_1$, whenever $\mu,\nu$ are probability measures it follows directly that $\dc(\mu,\nu) \le W_1(\mu,\nu)$, the $1$-Wasserstein distance defined via the theory of optimal transport, see~\cite{santambrogio2015optimal,ambrosio2021lectures} for its definition and basic properties. Therefore, a direct corollary is a stability result formulated w.r.t. $W_1$. Since the Wasserstein distances, on compact ambient spaces, metrize weak convergence of measures, estimate~\eqref{eq.BM_stability_prototype} provides a way of approximating general convex bodies with polygons, without any regularity or curvature assumption. 

Before commenting on the left-hand side of~\eqref{eq.BM_stability_prototype}, let us discuss why we should expect inequality~\eqref{eq.BM_stability_prototype} to hold. In finite-dimensional variational problems, and more generally in Hilbert spaces, quantitative stability of minimizers is typically a consequence of favorable coercivity properties of the functional being optimized~\cite{bonnans2013perturbation}. This mechanism is often captured by \textit{Łojasiewicz inequalities}~\cite{lojasiewicz1965ensembles}, which relate the energy gap to a distance of a point from the set of minimizers. A function 
$f$ satisfies a Łojasiewicz inequality if it holds for all $x$ that
\begin{equation}\label{eq.lojasiewicz_inequalities}
    \dist(x, \argmin f)^\theta 
    \leq  
    C(f(x) - \inf f)
\end{equation}
for some $\theta > 0$ and $C>0$.\footnote{
    For sufficiently smooth energies $f$, it is equivalent to having an inequality of the form ${(f(x) - \inf f)}^\theta \leq C \|\nabla f(x)\|$.
} 
In this sense, the usage of the Fraenkel asymetry in~\eqref{eq.BM_stability_prototype} reflects naturally these inequalities in the form of a distance to the set o maximizers of the functional $\mathscr{F}_\mu$ used in your proof, see Lemma~\ref{lemma.variational_problem} and also Remark~\ref{remark.lojasiewicz} following it. Such inequalities reflect an underlying curvature of the energy landscape and play a central role in stability and convergence results for associated gradient flows. Indeed, if $f$ is $\alpha$-strongly convex this identity is true with $\theta = 2$ and $C = \alpha^{-1}$. Extensions of this philosophy to non-linear spaces, such as the Wasserstein space of probability measures, have been developed in recent years, where functional inequalities are reinterpreted as generalized Łojasiewicz inequalities, see for instance~\cite{blanchet2018family}. 

The situation for the Minkowski problem is fundamentally different. The natural topologies on the space of convex bodies, such as Hausdorff convergence or $L^1$-convergence are essentially flat, so instead of leveraging curvature, stability must arise from the rigidity of convex geometry. In this case, convexity/concavity stems naturally from the classical \textit{Brunn-Minkowski inequality}, the reader is referred to~\cite{gardner2002brunn,schneider2014covexbodies,boroczky2025isoperimetric} for comprehensive surveys on the many applications of the Brunn-Minkowski theory. It can be stated as follows: given two convex bodies $E,F \subset \mathbb{R}^n$, for all $t \in (0,1)$, we have that
\begin{equation}\label{eq.Brunn-Minkowski}
    {|t E + (1-t)F|}^{\frac{1}{n}} 
    \ge 
    t{|E|}^{\frac{1}{n}} + (1-t){|F|}^{\frac{1}{n}},  
\end{equation}   
where equality holds if and only if $E$ and $F$ are equal up a translation and a dilation.  

In this context, we propose a variational problem whose optimizers are translations and dilations of Minkowski bodies associated with a prescribed measure $\mu$, that act as a perturbation parameter. For this energy the quantitative stability of the Brunn-Minkowski inequality, see for instance~\cite{figalli2009refined,figalli2017quantitative,figalli2024sharp}, assumes the role of a Łojasiewicz inequality and allows us to deduce stability for Minkowski bodies. This explains the exponent $2$ on the Fraenkel asymmetry term as the sharp exponent for the quantitative stability versions of both the Brunn-Minkowski and the Wulff isoperimetric inequalities, see for instance~\cite{fusco2008sharp,figalli2010stability,figalli2017quantitative}. The exponent on the right-hand side $1 + \frac{1}{n}$ is obtained by combining a gradient-concavity principle with an older quantitative version of isoperimetry due to Diskant~\cite{diskant1972bounds}. 

Previous results from Hug and Schneider~\cite{hug2002stability} obtain stability results conditioned to having bounds on the inner and circumradii of Minkowski bodies. This conditioning on geometric bounds in~\cite{hug2002stability} seem inevitable if we do not make quantitative the assumption of Alexandrov's theorem, that $\mu$ is not concentrated on a hyperplane, which is reflected in the constant $C$ from~\eqref{eq.BM_stability_prototype}. Notice that if such an inequality was true for any pair of measures $\mu,\nu$, we could take $\mu = \frac{1}{2}\delta_\theta + \frac{1}{2}\delta_{-\theta}$, for some $\theta \in {\mathbb{S}}^{n-1}$, a sequence ${\left(\mu_k\right)}_{k \in \mathbb{N}}$ not supported in a hyperplane and converging to $\mu$ in the Wasserstein distance. As $\mu_k$ is a Cauchy sequence for $W_1$, if~\eqref{eq.BM_stability_prototype} holds for all measures in $\mathscr{P}(\mathbb{S}^{n-1})$ so will be the associated sequence of Minkowski bodies $E_k$, which would allow to construct a Minkowski body for $\mu$. This is clearly absurd, since Alexandrov's condition of $\mu$ not being supported on a hyperplane is necessary and sufficient for the existence of a Minkowski body.

Therefore, to obtain a result such as~\eqref{eq.BM_stability_prototype}, we need to make quantitative the assumption of $\mu$ not being concentrated on a hyperplane. This idea has been exploited in a recent work of the second and third authors~\cite{machado2025quantitative} concerning the quantitative stability of moment measures through the following functional 
\begin{equation}\label{eq.Theta_functional}
    \Theta(\mu) 
    \eqdef 
    \inf_{\theta \in \mathbb{S}^{n-1}} 
    \int_{\mathbb{S}^{n-1}} |\theta \cdot x| \dd \mu(x).
\end{equation}
If follows directly from the definition that $\Theta(\mu) > 0$ if and only if $\mu$ is not supported on a hyperplane, and hence it can be used as quantitative measure of how this condition is satisfied. In addition, if $E_\mu$ is the centered Minkowski body associated to $\mu$, we have for all $\theta \in \mathbb{S}^{n-1}$ that
\begin{equation}
    \int_{\mathbb{S}^{n-1}} |\theta \cdot y| \dd \mu(y) 
    = 
    \int_{\partial E_\mu} |\theta \cdot \nu_{E_\mu}(x)| \dd \mathcal{H}^{n-1}(x)
    = 
    \mathcal{H}^{n-1}({\proj}_{\theta^\perp} E_\mu),
\end{equation}
where the last equality is known as \textit{Cauchy's projection formula}. As a result, the quantity $\Theta(\mu)$ gives a lower bound on the size of the projections of $E_\mu$ and hence on the size of $E_\mu$ itself. See Figure~\ref{fig.degenerate_minko_bodies} for an illustration of the role of this functional in the construction proposed in the previous paragraph. As a result, the constant $C$ in~\eqref{eq.BM_stability_prototype} will be uniform in the class of measures $\mu \in \mathscr{P}(\mathbb{S}^{n-1})$ such that $\Theta(\mu) \geq \vartheta$, and $C = C_{n,\vartheta}$.  As a result, the stability of Minkowski bodies degenerates as $\mu$ and $\nu$ approach measures supported on hyperplanes. The rigorous statement of our quantitative stability result is as follows. 

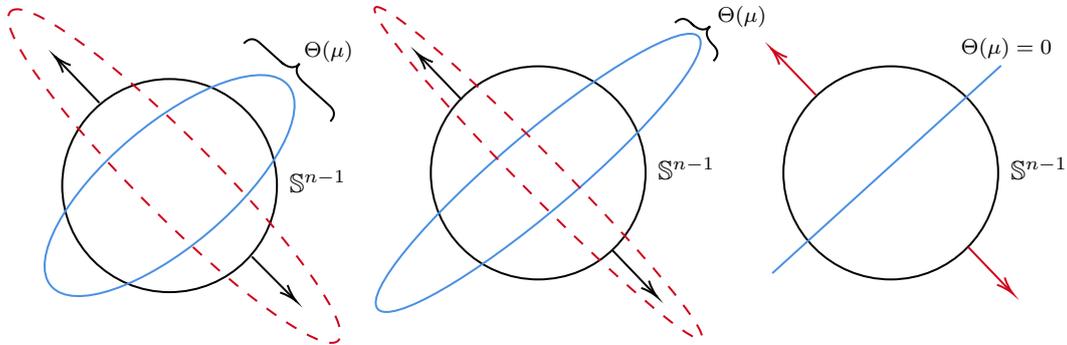
\begin{figure}[ht]\label{fig.degenerate_minko_bodies}
    \centering
    \tikzset{every picture/.style={line width=0.75pt}} %set default line width to 0.75pt        
\centering
\begin{tikzpicture}[x=0.75pt,y=0.75pt,yscale=-0.8,xscale=0.8]
%uncomment if require: \path (0,302); %set diagram left start at 0, and has height of 302

%Shape: Circle [id:dp5978115037546095] 
\draw   (35,143.33) .. controls (35,106.33) and (65,76.33) .. (102,76.33) .. controls (139,76.33) and (169,106.33) .. (169,143.33) .. controls (169,180.33) and (139,210.33) .. (102,210.33) .. controls (65,210.33) and (35,180.33) .. (35,143.33) -- cycle ;
%Shape: Ellipse [id:dp29341954877000376] 
\draw  [color={rgb, 255:red, 74; green, 144; blue, 226 }  ,draw opacity=1 ] (28.08,206.46) .. controls (14.4,190.44) and (36.4,149.19) .. (77.23,114.32) .. controls (118.05,79.46) and (162.24,64.18) .. (175.92,80.2) .. controls (189.6,96.21) and (167.6,137.47) .. (126.77,172.33) .. controls (85.95,207.2) and (41.76,222.48) .. (28.08,206.46) -- cycle ;
%Shape: Ellipse [id:dp23881267719319044] 
\draw  [color={rgb, 255:red, 208; green, 2; blue, 27 }  ,draw opacity=1 ][dash pattern={on 4.5pt off 4.5pt}] (2.91,33.73) .. controls (12.74,24.77) and (66.42,63.88) .. (122.81,121.1) .. controls (179.19,178.32) and (216.93,231.97) .. (207.09,240.93) .. controls (197.26,249.89) and (143.58,210.77) .. (87.19,153.55) .. controls (30.81,96.34) and (-6.93,42.69) .. (2.91,33.73) -- cycle ;
%Straight Lines [id:da7441815125115865] 
\draw    (58.23,91.32) -- (31.71,63.45) ;
\draw [shift={(30.33,62)}, rotate = 46.43] [color={rgb, 255:red, 0; green, 0; blue, 0 }  ][line width=0.75]    (10.93,-3.29) .. controls (6.95,-1.4) and (3.31,-0.3) .. (0,0) .. controls (3.31,0.3) and (6.95,1.4) .. (10.93,3.29)   ;
%Shape: Circle [id:dp7005638453799473] 
\draw   (265,135.33) .. controls (265,98.33) and (295,68.33) .. (332,68.33) .. controls (369,68.33) and (399,98.33) .. (399,135.33) .. controls (399,172.33) and (369,202.33) .. (332,202.33) .. controls (295,202.33) and (265,172.33) .. (265,135.33) -- cycle ;
%Shape: Ellipse [id:dp6845662037057292] 
\draw  [color={rgb, 255:red, 74; green, 144; blue, 226 }  ,draw opacity=1 ] (231.85,220.86) .. controls (223.25,210.79) and (261.12,164.34) .. (316.43,117.1) .. controls (371.74,69.86) and (423.55,39.73) .. (432.15,49.79) .. controls (440.75,59.86) and (402.88,106.32) .. (347.57,153.56) .. controls (292.26,200.8) and (240.45,230.93) .. (231.85,220.86) -- cycle ;
%Shape: Ellipse [id:dp8539648982603185] 
\draw  [color={rgb, 255:red, 208; green, 2; blue, 27 }  ,draw opacity=1 ][dash pattern={on 4.5pt off 4.5pt}] (230.29,31.35) .. controls (235.31,26.81) and (284.92,69.68) .. (341.09,127.11) .. controls (397.26,184.53) and (438.73,234.76) .. (433.71,239.3) .. controls (428.69,243.84) and (379.08,200.97) .. (322.91,143.55) .. controls (266.74,86.13) and (225.27,35.9) .. (230.29,31.35) -- cycle ;
%Straight Lines [id:da14020280087178227] 
\draw    (284.23,89.32) -- (257.71,61.45) ;
\draw [shift={(256.33,60)}, rotate = 46.43] [color={rgb, 255:red, 0; green, 0; blue, 0 }  ][line width=0.75]    (10.93,-3.29) .. controls (6.95,-1.4) and (3.31,-0.3) .. (0,0) .. controls (3.31,0.3) and (6.95,1.4) .. (10.93,3.29)   ;
%Straight Lines [id:da5603361679448644] 
\draw    (378.33,184) -- (404.85,211.87) ;
\draw [shift={(406.23,213.32)}, rotate = 226.43] [color={rgb, 255:red, 0; green, 0; blue, 0 }  ][line width=0.75]    (10.93,-3.29) .. controls (6.95,-1.4) and (3.31,-0.3) .. (0,0) .. controls (3.31,0.3) and (6.95,1.4) .. (10.93,3.29)   ;
%Shape: Circle [id:dp9475784237826396] 
\draw   (485,135.33) .. controls (485,98.33) and (515,68.33) .. (552,68.33) .. controls (589,68.33) and (619,98.33) .. (619,135.33) .. controls (619,172.33) and (589,202.33) .. (552,202.33) .. controls (515,202.33) and (485,172.33) .. (485,135.33) -- cycle ;
%Straight Lines [id:da6044181775537695] 
\draw [color={rgb, 255:red, 208; green, 2; blue, 27 }  ,draw opacity=1 ]   (505.23,86.32) -- (478.71,58.45) ;
\draw [shift={(477.33,57)}, rotate = 46.43] [color={rgb, 255:red, 208; green, 2; blue, 27 }  ,draw opacity=1 ][line width=0.75]    (10.93,-3.29) .. controls (6.95,-1.4) and (3.31,-0.3) .. (0,0) .. controls (3.31,0.3) and (6.95,1.4) .. (10.93,3.29)   ;
%Straight Lines [id:da29508541854137216] 
\draw [color={rgb, 255:red, 208; green, 2; blue, 27 }  ,draw opacity=1 ]   (600.33,182) -- (626.85,209.87) ;
\draw [shift={(628.23,211.32)}, rotate = 226.43] [color={rgb, 255:red, 208; green, 2; blue, 27 }  ,draw opacity=1 ][line width=0.75]    (10.93,-3.29) .. controls (6.95,-1.4) and (3.31,-0.3) .. (0,0) .. controls (3.31,0.3) and (6.95,1.4) .. (10.93,3.29)   ;
%Straight Lines [id:da38477923288259275] 
\draw [color={rgb, 255:red, 74; green, 144; blue, 226 }  ,draw opacity=1 ]   (478.08,198.46) -- (620.82,67.86) ;
%Straight Lines [id:da7861012215950758] 
\draw    (153.33,188) -- (179.59,214.58) ;
\draw [shift={(181,216)}, rotate = 225.34] [color={rgb, 255:red, 0; green, 0; blue, 0 }  ][line width=0.75]    (10.93,-3.29) .. controls (6.95,-1.4) and (3.31,-0.3) .. (0,0) .. controls (3.31,0.3) and (6.95,1.4) .. (10.93,3.29)   ;
%Shape: Brace [id:dp7772620483523939] 
\draw   (202.33,103) .. controls (205.47,99.55) and (205.31,96.25) .. (201.85,93.11) -- (185.8,78.55) .. controls (180.87,74.07) and (179.97,70.1) .. (183.1,66.64) .. controls (179.97,70.1) and (175.93,69.59) .. (170.99,65.11)(173.21,67.12) -- (158.22,53.52) .. controls (154.77,50.39) and (151.47,50.55) .. (148.33,54) ;
%Shape: Brace [id:dp11747602823300118] 
\draw   (441.33,65) .. controls (444.62,61.71) and (444.62,58.41) .. (441.33,55.12) -- (441.33,55.12) .. controls (436.62,50.41) and (435.92,46.41) .. (439.22,43.12) .. controls (435.92,46.41) and (431.92,45.71) .. (427.22,41)(429.33,43.12) -- (427.22,41) .. controls (423.92,37.71) and (420.62,37.71) .. (417.33,41) ;

% Text Node
\draw (175,132.06) node [anchor=north west][inner sep=0.75pt]    {$\mathbb{S}^{n-1}$};
% Text Node
\draw (405,124.06) node [anchor=north west][inner sep=0.75pt]    {$\mathbb{S}^{n-1}$};
% Text Node
\draw (625,124.06) node [anchor=north west][inner sep=0.75pt]    {$\mathbb{S}^{n-1}$};
% Text Node
\draw (184,50.4) node [anchor=north west][inner sep=0.75pt]  [font=\footnotesize]  {$\Theta ( \mu )$};
% Text Node
\draw (442,28.4) node [anchor=north west][inner sep=0.75pt]  [font=\footnotesize]  {$\Theta ( \mu )$};
% Text Node
\draw (594,48.4) node [anchor=north west][inner sep=0.75pt]  [font=\footnotesize]  {$\Theta ( \mu ) =0$};

\end{tikzpicture}
    \caption{Family of ellipsoids collapsing to a line as $\Theta$ goes to $0$.}
\end{figure}

\begin{theorem}\label{thm.quantitative_stability_p=1}
    Let $\mu, \nu \in \mathscr{P}(\mathbb{S}^{n-1})$ be two measures such that $\Theta(\mu), \Theta(\nu) \ge \vartheta >0$, and let $E_\mu, E_\nu$ be the associated Minkowski bodies centered at the origin. Then there is a constant $C_{\vartheta,n} >0$ depending only on $\vartheta$ and the dimension $n$ such that 
    \begin{equation}
        \inf_{x_0 \in \mathbb{R}^n} 
        \dH(E_\mu, x_0 + E_\nu) 
        \le C_{\vartheta,n} {\dc(\mu,\nu)}^{\frac{1}{n-1}},
    \end{equation}
    where the exponent $\frac{1}{n-1}$ is optimal, 
    and 
    \begin{equation}
        {\alpha(E_\mu, E_\nu)}^2
        \le 
        C_{\vartheta,n} {\dc(\mu, \nu)}^{1 + \frac{1}{n-1}},
    \end{equation}
    with a sharp exponent in dimension $2$.
\end{theorem}
% \noindent{\bf Remark. } It follows that if $\mu(S^{n-1})=\nu(S^{n-1})=1$, then
%  \begin{equation}
% \alpha(E_\mu, E_\nu)
%      \le 
%   C_{\vartheta,n} W_1(\mu, \nu)^{1/2}.
% \end{equation}

It turns out that our techniques are also applicable to the $L_p$-Minkowski problems for all $1 \le p \neq n$. Given a convex body $E$, Lutwak~\cite{lutwak1992brunn-minkowki-firey,lutwak1996brunn} defines the $L_p$ curvature measure as
\begin{equation}
    S_{E,p} \eqdef h_E^{1-p}S_E,
\end{equation}
and one can ask the same question as before, under which conditions on $\mu \in \mathscr{P}(\mathbb{S}^{n-1})$ can we find a convex body $E_{\mu,p}$ whose $L_p$ surface measure coincides with $\mu$? In the regime $1 \le p \neq n$, this problem is well understood, see for instance~\cite{cheng1976regularity,hug2005lp,boroczkylogarithmic,HYZ2025}. On the other hand, whenever $\dd\mu=f\dd(\mathcal{H}^{n-1}\mres \mathbb{S}^{n-1})$ the corresponding Monge-Ampère equation becomes  
\[
    \det(\nabla^2 h+h I)=h^{p-1}f. 
\]

In this case, the $L_p$ surface measure is not translation invariant, as translating a convex body changes its support function. Therefore, this $L_p$ Minkowski body is unique whenever it exists, instead of unique up to translations, and we no longer need to assume that the measure $\mu$ is centered at the origin. In fact, we have existence and uniqueness whenever $\mu$ is not concentrated on any half space, that is 
\begin{equation}\label{eq.half_space_assumption}
    \mu(H_\theta^+) > 0, \text{ for all $\theta \in \mathbb{S}^{n-1}$ where }
    H_\theta^+ \eqdef \{
        x \in \mathbb{R}^n : x\cdot \theta > 0
    \}.
\end{equation}
In a sense, this is the correct dispersion assumption for the $p = 1$ case a well, but with the further assumption that $\mu$ is centered, it becomes equivalent to assuming $\Theta(\mu) > 0$. Without the centering assumption on $\mu$ in the $L_p$ case, the natural functional to make~\eqref{eq.half_space_assumption} quantitative is 
\begin{equation}
    \Theta_+(\mu) 
    \eqdef 
    \inf_{\theta \in \mathbb{S}^{n-1}} 
    \int_{\mathbb{S}^{n-1}} 
    (\theta\cdot x)_+ \dd \mu(x). 
\end{equation}

Since we always have $\Theta(\mu) \ge \Theta_+(\mu)$, we deduce a general control of the geometry of a convex body $E$ by means of the quantity $\Theta(S_{E,p})$. Whereas, for the isoperimetric/Brunn-Minkowski rigidity, we exploit the fact that this variant is closely related to the $L_p$ Brunn-Minkowski theory started by 
Lutwak~\cite{lutwak1992brunn-minkowki-firey,lutwak1996brunn} extending earlier results by Firey~\cite{firey1962p,firey1974shapes}, see also the recent monograph covering the state of the art on this topic~\cite{boroczky2025isoperimetric}. In Section~\ref{sec.quantitativeLpBM_iso}, we leverage the quantitative stability for the Brunn-Minkowski and isoperimetric inequalities in the $p = 1$ case to obtain $L_p$ counterparts of such inequalities. This discussion gives us all the elements necessary to generalize the proof of Theorem~\ref{thm.quantitative_stability_p=1} to the $L_p$ setting.

\begin{theorem}\label{thm.quantitative_stability_p}
    Consider $1 < p \neq n$, and let $\mu, \nu \in \mathscr{P}(\mathbb{S}^{n-1})$ be two measures such that $\Theta_+(\mu), \Theta_+(\nu) \ge \vartheta >0$. Consider their associated $L_p$ Minkowski bodies $E_{\mu,p}, E_{\nu,p}$. Then there is a constant $C_{\vartheta,n,p} >0$ depending only on $\vartheta$, the dimension $n$, and the parameter $p$ such that 
    \begin{equation}
        \inf_{x \in \mathbb{R}^n} 
        \dH(E_{\mu,p}, x + E_{\nu,p}) 
        \le C_{\vartheta,n,p} {\dc(\mu,\nu)}^{\frac{1}{n}},
        \text{ and }
        {\alpha(E_{\mu,p}, E_{\nu,p})}^2
        \le 
        C_{\vartheta,n,p} {\dc(\mu, \nu)}.
    \end{equation}

    If in addition $\mu,\nu$ are centered at the origin, then we have a control in Hausdorff distance
    \begin{equation*}
        \inf_{x \in \mathbb{R}^n} 
        \dH(E_{\mu,p}, x + E_{\nu,p}) 
        \le C_{\vartheta,n,p} {\dc(\mu,\nu)}^{\frac{1}{n-1}},
    \end{equation*}
    and the stronger control in Fraenkel asymmetry
    \begin{equation*}
        {\alpha(E_{\mu,p}, E_{\nu,p})}^2
        \le 
        C_{\vartheta,n,p} {\dc(\mu, \nu)}^{1 + \frac{1}{n-1}},
    \end{equation*}
    where $C_{\vartheta,n,p}$ is a constant depending only on $\vartheta$, $n$, and $p$.
\end{theorem}

Concerning the structure of the paper, Section~\ref{sec.preliminaries} introduces the main notions used in the paper, as well as the control of the geometry of convex bodies with the functional $\Theta$ applied to its surface measures and $L_p$ surface measures. The stability versions Theorem~\ref{thm.quantitative_stability_p=1} of the classical Minkowski problem is proved in Section~\ref{sec.case_p1}, and Theorem~\ref{thm.quantitative_stability_p} of the $L_p$-Minkowski problem for $1< p \neq n$ is proved in Section~\ref{secp>1}.

\section*{Acknowledgments}
B\"or\"oczky was partially supported by NKKP grant 150613. J.M.M. would like to thank Guillaume Carlier and Quentin Mérigot for stimulating discussions on Minkowski problems. He also warmly thanks the hospitality of IMPA, where the majority of this work was prepared while visiting J.P.G.R. during the summer program. The research of J.P.G.R. is supported by the Portuguese government through FCT - Fundação para a Ciência e a Tecnologia, I.P., project 2023.17881.ICDT with DOI identifier 10.54499/2023.17881.ICDT (project SHADE). 

\section{Preliminaries on convex bodies and Minkowski problems}\label{sec.preliminaries}

By a convex body $K \subset \mathbb{R}^n$ we mean a closed and convex subset of $\mathbb{R}^n$ with non-empty interior. 
We say that a convex body $K$ is centered at a certain point $x_0$ whenever its barycenter lies that $x_0$. Simply, we say $K$ is centered when its barycenter is the origin, that is if $\displaystyle \int_K x \dd x = 0$. Convex bodies are uniquely determined by their support function, given by
\begin{equation}\label{eq.support_function}
    h_K(x) \eqdef \sup_{y \in K} x \cdot y, \text{ for $x \in \mathbb{S}^{n-1}$,}
\end{equation}
and then extended uniquely to a $1$-homogenous function. It can be easily checked that $K = \left\{ x \in \mathbb{R}^n: x\cdot y \le h_K(y) \text{ for all } y \in \mathbb{S}^{n-1} 
\right\}$, so it encodes many of its geometric properties.

A central object to the present work and to the Brunn-Minkowski theory of convex bodies is the surface area measure. Given a convex body $K$, we can define its surface measure as the pushforward of the $(n-1)$-dimensional Hausdorff measure on $\partial K$ through the Gauss map $\nu_K : \partial K \mapsto \mathbb{S}^{n-1}$, that is 
\[
    S_K \eqdef {(\nu_K)}_\sharp \mathcal{H}^{n-1}\mres \partial K. 
\]
Equivalently, for a given Borel set $\omega \subset \mathbb{S}^{n-1}$ it can be written as 
\[
    S_K(\omega) = \mathcal{H}^{n-1}(\nu_K^{-1}(\omega)). 
\] 
At any point $x\in\partial K$ such that the normal vector is unique, it coincides with $\nu_K(x)$, so that a surface area measure $S_K$ can be viewed as the distribution of normal vectors associated with the convex body $K$. 

These measures naturally arise as the first variation of the volume functional, or the mixed volumes of convex bodies, see for instance~\cite{boroczky2025isoperimetric}. In particular, if $K$ and $L$ are convex bodies, then the following identity holds
\begin{equation}\label{eq.first_variation_volume} 
    \frac{\dd}{\dd t} \Big|_{t=0} |K + tL| 
    =
    \int_{\mathbb{S}^{n-1}} h_L(x) \dd S_K(x).
\end{equation}
and we define the mixed volume of $K$ and $L$ as
\begin{equation}
    V_1(K,L) 
    \eqdef 
    \frac{1}{n} \int_{\mathbb{S}^{n-1}} h_L(x) \dd S_K(x).
\end{equation}
Choosing $K=L$, we obtain that $V(K,K) = |K|$, and it follows from the Brunn-Minkowski inequality~\eqref{eq.Brunn-Minkowski} that 
\begin{equation}\label{eq.Wulff_inequality}
    V_1(K,L) \ge {|K|}^{\frac{n-1}{n}}{|L|}^{\frac{1}{n}}.
\end{equation}
This identity is known by many names, such as the \textit{Wulff isoperimetric inequality}, or the \textit{anisotropic isoperimetric inequality}, and it can be seen as a generalization of the classical isoperimetric inequality, which corresponds to the case $L = B_1(0)$, where $B_1(0)$ is the unit ball centered at the origin.

These identities are fundamental tools to control the geometry of convex bodies, which is done by introducing some quantitative measures of their size. For instance, defining the \textit{inradius} and \textit{circumradius} of a centered convex body $K$ as 
\begin{equation}\label{eq.in_circum_radii}
        r_K
        \eqdef 
        \sup
        \left\{
            r > 0 : B_r(0) \subset K        
        \right\}, \
        R_K 
        \eqdef 
        \inf
        \left\{
            r > 0 : K \subset B_r(0)        
        \right\},
\end{equation}
if follows, whenever $0 \in \intt K$, that $0 < r_K \le h_K \le R_K$. 

A fundamental tool to control the these quantities is the John ellipsoid, which is the unique ellipsoid of maximal volume contained in $K$. It can be shown that if $E$ is the John ellipsoid of $K$, then
$
    E \subset K \subset n E.
$ Letting $a_1 \le \dots \le a_n$ be the half-axes of $E$, ordered in ascending order so that $a_1$ is the smallest half-axis, and $a_n$ the largest one, we have that 
\[
        \frac{r_K}{n} \le a_1 \le r_K \text{ and } a_n \le R_K \le n a_n.
\]

The rest of this section is dedicated to obtaining upper bounds on $R_K$ and lower bounds on $r_K$ based on estimates of the form $\Theta(\mu) \ge \vartheta$, where $\mu$ is a geometric measure associated with a convex body $K$, for instance the usual surface measure $S_K$ discussed above, or the $L_p$-surface measure, relevant to the $L_p$-Minkowski problem, for $p>1$, see Section~\ref{sec.technical_estimates_p>1}. 

\subsection{Inner and circumradii estimates with $\Theta(S_K)$}\label{sec.technical_estimates_p=1}

As discussed in the introduction, the functional defined over the set of Radon measures on the sphere 
\[
    \mathscr{M}(\mathbb{S}^{n-1}) \ni \mu 
    \mapsto \Theta(\mu) 
    \eqdef 
    \inf_{\theta \in \mathbb{S}^{n-1}} 
    \int_{\mathbb{S}^{n-1}} 
    |\theta\cdot x|\dd \mu(x)
\]
can be used to obtain quantitative estimates on the geometry of a convex body $K$. Indeed, thanks to the Cauchy projection formula, we have that given $\theta \in \mathbb{S}^{n-1}$, let ${\proj}_{\theta^\perp}$ denote the projection operator onto the hyperplane orthogonal to $\theta$, then it holds the following lower bound on the projection area of $K$ onto $\theta^\perp$:
\begin{equation}\label{eq.projection_area_lower_bound}
    \mathcal{H}^{n-1}({\proj}_{\theta^\perp}(K)) 
    =
    \int_{\partial K} 
    |\theta\cdot \nu_K|\dd \mathcal{H}^{n-1} 
    =
    \int_{\mathbb{S}^{n-1}} 
    |\theta\cdot x|\dd S_K(x)
    \ge \Theta(S_K).
\end{equation}
This can be used to bound the John ellipsoid of $K$ and hence to obtain estimates on the inner and circumradii of $K$ in terms of $\Theta(S_K)$ and $S(K)$.
\begin{lemma}\label{lemma.rR}
For a convex body $K$ in $\R^n$ there exist dimensional constants $C_n,c_n > 0$ such that
\[
    c_n
    \frac{\Theta(S_K)}{{S(K)}^{\frac{n-2}{n-1}}} \le r_K \leq R_K
    \le 
    C_n 
    \frac{S(K)}{{\Theta(S_K)}^{\frac{n-2}{n-1}}},
\] 
where $c_n = 1/{2{n}^{n^{(n-2)/(n-1)}}}$ and $C_n = \frac{n}{2}{n}^{n^{(n-2)/(n-1)}}$.
\end{lemma}
\begin{proof}
For simplicity of notation let $r=r_K$ and $R=R_K$ denote the inner and circumradii of $K$. Since translations of the convex body do not change the surface measure, we may assume that the John ellipsoid $E$ of $K$ is centered at the origin. Let $a_1\leq\cdots\leq a_n$ denote the length of the half axes of $E$. Since $E\subset K\subset nE$, we have 
\[
    (r/n)\leq a_1\leq\cdots\leq a_n\leq R \mbox{ \ and \ }a_1\leq r.
\]

Choosing the direction $\theta_n$ associated with the largest half-axis $a_n$, from the inclusion the inclusion $K\subseteq nE$ we have that $\proj_{\theta_n^\perp}(K) \subset Q_n \eqdef n\left([-a_1, a_1]\times \cdots\times [-a_{n-1}, a_{n-1}]\right)$. Hence, it follows from the monotonicity of the perimeter that 
\begin{align*}
    \Theta(S_K) 
    &\le 
    \mathcal{H}^{n-1}({\proj}_{\theta_n^\perp}(K)) 
    \le 
    \mathcal{H}^{n-1}
    \left(
        Q_n
    \right)\\
    &= 
    {(2n)}^{n-1}a_1\cdots a_{n-1}
    \le 
    {(2n)}^{n-1}r R^{n-2}.
\end{align*}
On the other hand, choosing the direction $\theta_1$ associated with the smallest half-axis $a_1$, the inclusion $E\subseteq K$ gives that $Q_1 \eqdef n\left( [-a_2, a_2]\times\cdots\times [-a_n,a_n]\right)$. So recalling that $r/n \le a_i$ for all $i=2,\dots,n$, and $R/n \le an$, it follows again from the monotonicity of the perimeter that
\begin{align*}
    {\left(\frac{2}{n}\right)}^{n-1} r^{n-2}R
    &\le 
    2^{n-1}a_2\cdots a_n 
    =
    \mathcal{H}^{n-1}\left(Q_1\right)
    \le
    \mathcal{H}^{n-1}\left({\proj}_{\theta_1^\perp}(K)\right) 
    \le
    S(K).
\end{align*}
Dividing these identities, using that $r \le a_1$, and $R \le a_n/n$, we obtain that 
\[  
    1 \le 
    \frac{R}{r} 
    \le 
    \frac{na_n}{a_1} 
    = 
    \frac{n a_n\cdots a_2}{a_1\cdots a_{n-1}} 
    \le 
    n^n
    \frac{S(K)}{\Theta(S_K)}.
\]

Plugging this estimate back into the previous inequalities, we obtain the conclusion. 
\end{proof} 

\subsection{On the $L_p$-Minkowski problem for $p>1$}\label{sec.technical_estimates_p>1}
Given a convex body $K$ in $\R^n$ with $o\in K$, we can also define the \textit{$L_p$-surface measures} as 
\begin{equation}
    \dd S_{K,p}=h_K^{1-p} \dd S_K. 
\end{equation}
Letting $\partial' K$ denote the set of ``regular'' points $x\in \partial K$ where there exists a unique exterior unit normal $\nu_K(x)\in \mathbb{S}^{n-1}$, we have that for each $x \in \partial' K$ and $y\in K$
\begin{equation}
\label{nuKx-def}
h_K(x) = x\cdot \nu_K(x)\geq y \cdot \nu_K(x).
\end{equation}
It also holds that $\mathcal{H}^{n-1}(\partial K\backslash \partial' K)=0$, and as a result, if $\omega\subset \mathbb{S}^{n-1}$ is Borel, the $L_p$-surface measure of $\omega$ can be written as
\begin{equation}
\label{Lp-surf-boundary}
S_{K,p}(\omega)=\int_{\nu_K^{-1}(\omega)}\big(x\cdot \nu_K(x)\big)^{1-p}\,d\mathcal{H}^{n-1}(x).
\end{equation}

The basic tool to deal with $L_p$ surface measures is Firey's $L_p$ Brunn-Minkowski theory~\cite{firey1962p}. It is strongly based on the equivalence between a convex body $K$ and its support function $h_K$. In analogy with the identity $h_{K + L} = h_K + h_L$ from the $p = 1$ case, given two convex bodies $K,L$ containing the origin, and two positive numbers $t,s > 0$, one can define their $L_p$ sum $t\cdot K +_p s\cdot L$ as the convex body whose support function is given by
\begin{equation}
    h_{t\cdot K +_p s\cdot L} \eqdef 
    {\left(
        t h_K^p + s h_L^p
    \right)}^{1/p}.
\end{equation}
Notice that, $(K,L,t,s) \mapsto t\cdot K +_p s \cdot L$ is operation depending on the specific choice of each variable, we cannot replace $t\cdot K$ by the dilated body $tK$ in the usual sense. Defining the support function $h_{t\cdot K +_p s\cdot L}$ as above is equivalent to defining the $L_p$ Minkowski sum as 
\begin{equation}
    t\cdot K +_p s\cdot L
    \eqdef 
    \left\{
        x \in \mathbb{R}^n: 
        \inner{x,u}
        \le 
        {\left( t h_K^p(u) + s h_L^p(u) \right)}^{1/p}, \text{ for all }
        u \in \mathbb{S}^{n-1}
    \right\}
\end{equation}

Leveraging the classical Brunn-Minkowski inequality, in the case $p = 1$, it can be shown that a suitable power of the volume functional is concave with respect to the $L_p$ convex combination. Namely, for any pair of convex bodies $K,L$ containing the origin, and $t,s\ge 0$, it holds that 
\begin{equation}\label{eq.LpBM}
    {|t\cdot K +_p s \cdot L|}^{p/n} \ge 
    t{|K|}^{p/n} + s{|L|}^{p/n}. 
\end{equation}

If before, the first variation of the volume functional, with variations generated by the usual Minkowski sum, gave rise to the usual surface measure $S_K$, performing $L_p$ variations the first variation of the volume functional gives rise to the $L_p$ surface measure $S_{K,p}$. More precisely, for any pair of convex bodies $K,L$ containing the origin, and any $t > 0$, it holds that
\begin{equation}\label{eq.first_variation_volume_Lp}
    \frac{\dd}{\dd t} \Big|_{t=0} |K +_p t\cdot L| 
    =
    \int_{\mathbb{S}^{n-1}} h_L^p(x) \dd S_{K,p}(x).
\end{equation}

Mimicking the definition of the mixed volume, we can define the $L_p$-mixed volume of $K$ and $L$ as
\begin{equation}
    V_p(K,L)
    \eqdef
    \frac{p}{n} \int_{\mathbb{S}^{n-1}} h_L^p(x) \dd S_{K,p}(x).
\end{equation}
As a consequence, we have the analogous volume formula in terms of the $L_p$ surface area measure 
\begin{equation}
    \int_{\mathbb{S}^{n-1}} h_K^p(x) \dd S_{K,p}(x) = \frac{n}{p} |K|,
\end{equation}
and with a similar argument from the $p=1$ case using the $L_p$ Brunn-Minkowski inequality, the $L_p$-Wulff isoperimetric inequality also holds. Namely, for any pair of convex bodies $K,L$ containing the origin, it holds that
\begin{equation}\label{eq.LpWulff}
    V_p(K,L) \ge {|K|}^{\frac{n-p}{n}}{|L|}^{\frac{p}{n}}.
\end{equation}

The $L_p$ Minkowski problem for $p>1$ then becomes the question of existence and uniqueness of a convex body $K$ such that $S_{K,p} = \mu$, for a given measure $\mu \in \mathscr{P}(\mathbb{S}^{n-1})$. Differently from the case $p=1$, the $L_p$ surface measure is not translation invariant. So we cannot assume $o\in{\rm int}\, K$, and we have to consider the case $o\in\partial K$. As a result, for the condition $d\mu=h_K^{1-p}\,dS_K$ to make sense, we must check that if $o\in \partial K$, then $S_K(\{h_K=0\})=0$. 

Readily, the following conditions are equivalent for a convex body $K\subset \mathbb{R}^n$ with $o\in\partial K$:
\begin{itemize}
\item $S_K(\{h_K=0\})=0$;
\item $S_K(N_K(o)\cap \mathbb{S}^{n-1})=0$, where $N_K(o)=\{y\in\mathbb{R}^n:y\cdot x\leq 0\mbox{ for }x\in K\}$ is the exterior normal cone at $o\in\partial K$;
\item $\mathcal{H}^{n-1}(\{x\in \partial' K:\nu_K(x)\cdot x=0\})=0$.
\end{itemize}

\noindent{\bf Remark. } The condition $S_K(\{h_K=0\})=0$ is a consequence if we write the Monge-Amp\`ere equation in the form
$$
\dd S_K=h_K^{p-1}\,\dd \mu,
$$
or in the traditional form
$$
\det(\nabla^2 h+hI)=h^{p-1}f.
$$

In particular, the solution of the $L_p$-Minkowski problem for $p>1$ reads as follows. 

\begin{theorem}\label{thm.existence_Minko_Lp}
Let $\mu$ be a finite Borel measure on $\mathbb{S}^{n-1}$. Then $\mu=S_{K,p}$; namely, the $L_p$ surface area measure of a convex body $K$ with $o\in K$ and $S_K(\{h_K=0\})=0$ if and only if $\mu$ is not concentrated on any closed hemisphere.
\end{theorem}

The hypothesis that $\mu$ is not concentrated on any closed hemisphere is equivalent saying that
\begin{equation}\label{eq.Theta_plus}
    \Theta_+(\mu) 
    \eqdef 
    \inf_{\theta \in \mathbb{S}^{n-1}}\int_{\mathbb{S}^{n-1}}{(\theta\cdot v)}_+\,dv>0.
\end{equation}
This restriction is strictly stronger than asking for a lower bound on the functional $\Theta(\mu)$, defined in~\eqref{eq.Theta_functional}. Indeed, if $\Theta_+(\mu) \ge \vartheta$, then $\Theta(\mu) \ge 2\vartheta$, since we can write 
\[
    \int_{\mathbb{S}^{n-1}} |\theta \cdot v| \dd \mu(v)
    = 
    \int_{\mathbb{S}^{n-1}} {(\theta \cdot v)}_+ \dd \mu(v)
    + 
    \int_{\mathbb{S}^{n-1}} {((-\theta) \cdot v)}_+ \dd \mu(v) 
    \ge 2\vartheta.
\]
As a result, the class of measures that admit a solution to the $L_p$-Minkowski problem for $p>1$ is strictly smaller than for the $p = 1$ case. On the other hand, obtaining bounds on the inner and circumradii of a convex body $K$ in terms of $\Theta(S_{K,p})$ is stronger than obtaining such bounds in terms of $\Theta_+(S_{K,p})$.

The following Lemma about the functionals $\Theta$ and $\Theta_+$ will be useful in the sequel.
\begin{lemma}\label{lemma.Theta_plus_Theta}
    Let $\mu$ be a finite Borel measure on $\mathbb{S}^{n-1}$. Define the sets 
    \[
        \Omega_{\theta,\delta}=\{v\in \mathbb{S}^{n-1}: |\theta\cdot v| > \delta\}.
    \] 
    Then, if $\Theta(\mu) > 0$, it holds that $\mu(\Omega_{\theta,\delta}) > \delta$ for all $\theta \in \mathbb{S}^{n-1}$, where $\delta$ is given by $\delta = \min\{\frac{\Theta(\mu)}{2\mu(\mathbb{S}^{n-1})},\frac{\Theta(\mu)}{2}\}$.
\end{lemma}
\begin{proof}
    Fix $\theta \in \mathbb{S}^{n-1}$ and consider $\delta = \min\{\frac{\Theta(\mu)}{2\mu(\mathbb{S}^{n-1})},\frac{\Theta(\mu)}{2}\}$. We can write
    \[
        0 <\Theta(\mu) 
        \le 
        \int_{\mathbb{S}^{n-1}} |\theta\cdot v| \dd \mu(v)
        \le 
        \delta \mu(\mathbb{S}^{n-1}) + \mu(\Omega_{\theta,\delta}).
    \]
    As a result, we have 
    $
        \mu(\Omega_{\theta,\delta}) 
        \ge 
        \Theta(\mu) - \delta \mu(\mathbb{S}^{n-1}) 
        \ge 
        \frac{\Theta(\mu)}{2} \ge \delta,
    $ and the results follows. Naturally, the same conclusion holds if we replace $\Theta(\mu)$ by $\Theta_+(\mu)$, and $\Omega_{\theta,\delta}$ by $\Omega^+_{\theta,\delta} = \{v\in \mathbb{S}^{n-1}: (\theta\cdot v)_+ > \delta\}$.
\end{proof}

Next, we will show that if $\Theta(S_{K,p})$ is bounded from below, then the inner and circumradii of $K$ are bounded from below and above, respectively. We start with the case $1 < p < n$. 

\begin{lemma}\label{lemma.rR_Lp<n}
    Given $1<p<n$, consider a convex body $K$ such that $o \in K$, $S_K(\{h_K=0\})=0$, $\Theta(S_{K,p}) > 0$, and normalized so that 
    \[
        S_{K,p}(\mathbb{S}^{n-1}) = 1.
    \]
    Then it holds that
    \[
        c_{n,p} \Theta(S_{K,p})^{\underline{\beta}} 
        \le r_K \le R_K 
        \le C_{n,p} \Theta(S_{K,p})^{-\overline{\beta}},
    \] 
    for some constants $c_{n,p},C_{n,p} > 0$ and exponents $\underline{\beta},\overline{\beta} > 0$ depending only on $n,p$. In particular, we have $\overline{\beta} = \frac{p}{p-1}$ and $\underline{\beta} = \frac{(n-1)np}{(p-1)(n-p)}$.
\end{lemma}
\begin{proof} We start with the following application of H\"older's inequality: note that 
\[
\qquad
\int_{\mathbb S^{n-1}} |\langle x,\theta\rangle|\, \dd S_K(x)
=
\int_{\mathbb S^{n-1}}
|\langle x,\theta\rangle|\,
h_K(x)^{-\frac{p-1}{p}}\,
h_K(x)^{\frac{p-1}{p}}\,
\dd S_K(x).
\]
Using H\"older's inequality, this yields that the left-hand side above is bounded by
\[
\left(
\int_{\mathbb S^{n-1}}
|\langle x,\theta\rangle|\,
h_K(x)^{1-p}\,\dd S_K(x)
\right)^{\frac1p}
\left(
\int_{\mathbb S^{n-1}}
|\langle x,\theta\rangle|\,
h_K(x)\,
\dd S_K(x)
\right)^{\frac{p-1}{p}} .
\]

\[
=
\left(
\int_{\mathbb S^{n-1}}
|\langle x,\theta\rangle|\, \dd S_{K,p}(x)
\right)^{\frac1p}
\left(
\int_{\mathbb S^{n-1}}
|\langle x,\theta\rangle|\,
h_K(x)\,
\dd S_K(x)
\right)^{\frac{p-1}{p}}
\]

\[
\le
S_{K,p}(\mathbb S^{n-1})^{1/p}
\left(
\int_{\mathbb S^{n-1}}
|\langle x,\theta\rangle|\,
h_K(x)\,
\dd S_K(x)
\right)^{\frac{p-1}{p}} .
\]
Integrating both sides in $\theta \in \mathbb{S}^{n-1}$, and using that $S_{K,p}(\mathbb S^{n-1}) = 1$, we have 
\[
\int_{\mathbb S^{n-1}}
\left(
\int_{\mathbb S^{n-1}}
|\langle x,\theta\rangle|\, \dd S_K(x)
\right)
\dd \mathcal{H}^{n-1}(\theta)\]
\[
\le
S_{K,p}(\mathbb S^{n-1})^{1/p}
\int_{\mathbb S^{n-1}}
\left(
\int_{\mathbb S^{n-1}}
|\langle x,\theta\rangle|\, h_K(x)\, \dd S_K(x)
\right)^{\frac{p-1}{p}}
\dd \mathcal{H}^{n-1}(\theta).
\]

Now, by using H\"older's inequality on the right-hand side and Fubini's theorem on the left-hand side, we get that 
\begin{align}\label{eq.upper-bound-perimeter}
    \sigma_{n-1}\, S(K)
    & \le 
    S_{K,p}(\mathbb S^{n-1})^{1/p}
    \sigma_{n-1}^{1/p}\,
    \left(
    \int_{\mathbb S^{n-1}}
    \left(
        \int_{\mathbb S^{n-1}} 
        |\langle x,\theta\rangle|\, h_K(x)\, \dd S_K(x)
    \right)^{\frac{p}{p-1} \cdot \frac{p-1}{p}}
    \, \dd \mathcal{H}^{n-1}(\theta)
    \right)^{\frac{p-1}{p}}  \cr 
    & = 
    S_{K,p}(\mathbb S^{n-1})^{1/p}
    \sigma_{n-1}
    \left(
        \int_{\mathbb S^{n-1}} h_K\, \dd S_K
    \right)^{\frac{p-1}{p}} 
    = \sigma_{n-1}n^{\frac{p-1}{p}}\, V(K)^{\frac{p-1}{p}},
\end{align} 
where $\sigma_{n-1} = \mathcal{H}^{n-1}(\mathbb{S}^{n-1})$ is the surface area of the unit sphere in $\mathbb{R}^n$. 

Next, we bound the perimeter with a power of the volume. Using the $L_p$-Wulff inequality~\eqref{eq.Lp_Minkowski_inequality} applied to the convex body $K$ and with $F$ being the unit ball, we obtain
\[
    1
    = 
    S_{K,p}(\mathbb{S}^{n-1}) 
    = 
    \int_{\mathbb{S}^{n-1}} h_K^{1-p} \, \dd S_{K} 
    \ge  
    n\,\omega_n^{\frac{p}{n}}\, {V(K)}^{\frac{n-p}{n}}.  
\]
Since we have supposed that $p \in (1,n)$, it follows at once that the volume of $K$ is bounded by an absolute constant times the $L_p$ perimeter of $K$. Plugging this back into~\eqref{eq.upper-bound-perimeter}, it follows that the perimeter of $K$ is bounded by an absolute constant 
\begin{equation}\label{eq.bound_perimeter_Lp_case}
    S(K)
    \le 
    c_{n,p},
    \text{ where $c_{n,p}$ depends only on $n$ and $p$.}
\end{equation}

Now, we note that since $\Theta(S_{K,p}) > 0$, taking  $\delta = \min\left\{\frac{1}{2},\frac{1}{2}\Theta(S_{K,p})\right\}$ it holds that 
\[
    \inf_{u \in \mathbb{S}^{n-1}} \int_{\{ x \in \mathbb{S}^{n-1} \colon |\langle x, u \rangle| \ge \delta\}} |\langle x,u\rangle| \, \dd S_{K,p}(x) > \delta. 
\]
Since $S_{K,p}(\mathbb{S}^{n-1}) = 1$, a simple inspection gives that $\delta = \frac{1}{2}\Theta(S_{K,p})$. 

Letting $R_K$ and $r_K$ denote the circumradius and inradius of $K$, respectively, consider $u_0 \in \mathbb{S}^{n-1}$ to be a direction on the sphere such that $R_K \cdot u_0 \in \overline{K}$. Hence, if $x \in \mathbb{S}^{n-1}$ belongs to the set where $|\langle x, u_0\rangle| \ge \delta$, it follows by the previous discussion that $h_K(x) \ge \delta \cdot R_K$. Hence, 
\begin{align}\label{eq.upper-R_k-per}
    \delta 
    & \le 
    \int_{
        \{x \in \mathbb{S}^{n-1} \colon |\langle x,u_0\rangle| \ge \delta\}
    } 
    |\langle x, u_0\rangle| \, 
    \dd S_{K,p}(x) \cr 
    & = 
    \int_{
        \{x \in \mathbb{S}^{n-1} \colon |\langle x,u_0\rangle| \ge \delta\}
    } |\langle x,u_0\rangle| \, h_{K}^{1-p} \, 
    \dd S_K(x) \cr 
    & \le \delta^{1-p} R_K^{1-p} S_K(\{x \in \mathbb{S}^{n-1} \colon |\langle x,u_0\rangle| \ge \delta\}). 
\end{align}

We now estimate the perimeter above in two different ways, both starting out by estimating 
\[
S_K(\{x \in \mathbb{S}^{n-1} \colon |\langle x,u_0\rangle| \ge \delta\}) \le S(K). 
\]
Indeed, from~\eqref{eq.bound_perimeter_Lp_case} we know that $S(K)$ is bounded by a constant depending only on $n$ and $p$. Secondly, we have that, as $K$ is contained in a ball of radius $R_K$, 
\[
S(K) \le \sigma_{n-1} \cdot R_K^{n-1}. 
\]

Since $1 <p < n$, the first of those inequalities entails in an \emph{upper bound} on $R_K$. More precisely, combing~\eqref{eq.upper-R_k-per} with~\eqref{eq.bound_perimeter_Lp_case}, we get the upper bound on $R_K$ 
\begin{equation}
    R_K 
    \le 
    C_{n,p}
    \Theta(S_{K,p})^{-\overline{\beta}},
\end{equation}
where it can be checked that the exponent is given by $\overline{\beta} = \frac{p}{p-1}$. 

On the other hand, the easy estimate $S(K) \le \sigma_{n-1} R_K^{n-1}$ combined with~\eqref{eq.upper-R_k-per} gives the lower bound on $R_K$
\begin{equation}\label{eq.lower_bound_R_Lp}
    R_K \ge
    \bar c_{n,p}
    \Theta(S_{K,p})^{\frac{p}{n-p}},
\end{equation}
for some constant $\bar c_{n,p}$ depending only on $n$ and $p$. Combining~\eqref{eq.lower_bound_R_Lp},\eqref{eq.upper-R_k-per}, and~\eqref{eq.upper-bound-perimeter} we get a lower bound on the volume of $K$ depending only on $n$, $p$ and $\Theta(S_{K,p})$, namely there is another constant $\tilde{c}_{n,p}$ such that
\begin{equation}
    V(K) \ge 
    \tilde{c}_{n,p}\, \Theta(S_{K,p})^{\frac{n-1}{n-p}\frac{p^2}{p-1}},
\end{equation}

Finally, in order to finish the proof, we note that, by taking the John ellipsoid  $E$ of $K$, with $a_1 \le \cdots \le a_n$ being its half axes, then 
\[
    2^n \cdot r_K \cdot R_K^{n-1} \ge 2^n \cdot a_1 \cdot \dots \cdot a_n \ge |K| \ge \tilde{c}_n. 
\]
This, together with the upper bound we have on $R_K$, implies that $r_K$ admits a lower bound of the form 
\[
    r_K \ge c_{n,p} \Theta(S_{K,p})^{\underline{\beta}},
\]
where $c_{n,p}$ is a constant depending only on the dimension, $p$ and the exponent $\underline{\beta}$ can be checked to be $\underline{\beta} = \frac{(n-1)np}{(p-1)(n-p)}$. This finishes the proof. 
\end{proof}

Some observations concerning the previous Lemma~\ref{lemma.rR_Lp<n} in comparison with Lemma~\ref{lemma.rR} for the case $p=1$ are in order.

\begin{remark}
    Note that all the constants can be made explicit and similar bounds can be deduced for general convex bodies, without the normalization $S_{K,p}(\mathbb{S}^{n-1}) = 1$. This simplification is very natural for our quantitative stability result and was made in order to keep the presentation simpler, while still conveying the dependency on $\Theta(S_{K,p})$, which is the quantity that 
    controls the geometry of $K$. 
\end{remark}

\begin{remark}
    Notice that the exponents of $\Theta(S_{K,p})$ do not converge to the exponent of $\Theta(S_K)$ in Lemma~\ref{lemma.rR} as $p \to 1$. This is not surprising, and is an artifact of the $L_p$ surface measure not being translation invariant. It is not clear however if the exponents are optimal.
\end{remark}

In the case $p > n$, the above arguments are greatly simplified, since the corresponding $L_p$-Wulff inequality gives directly a lower bound on the volume of any convex body such that the $L_p$-perimeter is normalized to $1$. For our purposes in quantitative stability the lower bound on the inner radius is only used to bound the volume from below, so this step can be by-passed completely. But since we expect this result to be useful in broader situations, we include a strategy to obtain quantitative lower bounds on the inner radius as well.  

\begin{lemma}\label{lemma.rR_Lp>n}
    Given $p>n$, consider a convex body $K$ such that $o \in K$, $S_K(\{h_K=0\})=0$, $\Theta(S_{K,p}) > 0$, and normalized so that 
    \[
        S_{K,p}(\mathbb{S}^{n-1}) = 1.
    \]
    Then its volume is bounded from below by a universal constant $|K| \ge V_{n,p} \eqdef n^{\frac{n}{p-n}} \omega_n^{\frac{p}{p-n}}$, and we have the following bounds 
    \[
        c_{n,p} \Theta(S_{K,p})^{p\frac{n-1}{p-n}}
        \le 
        r_K \le R_K 
        \le C_{n,p} \Theta(S_{K,p})^{-\frac{p}{p-n}}.
    \]
\end{lemma}
\begin{proof}
    First we notice that an easy application of the $L_p$-Wulff inequality~\eqref{eq.LpWulff} gives that 
    \[ 
    1 = S_{K,p}(\mathbb{S}^{n-1}) \ge n \omega_n^{\frac{p}{n}} 
    |K|^{\frac{n-p}{n}}, 
    \] 
    so as $p > n$, we have that the volume of $K$ is bounded from below by the constant stated above. 

    Now, we note that since $\Theta(S_{K,p}) > 0$, taking  $\delta = \min\left\{\frac{1}{2},\frac{1}{2}\Theta(S_{K,p})\right\}$ it holds that 
\[
    \inf_{u \in \mathbb{S}^{n-1}} \int_{\{ x \in \mathbb{S}^{n-1} \colon |\langle x, u \rangle| \ge \delta\}} |\langle x,u\rangle| \, \dd S_{K,p}(x) > \delta. 
\]
Since $S_{K,p}(\mathbb{S}^{n-1}) = 1$, a simple inspection gives that $\delta = \frac{1}{2}\Theta(S_{K,p})$.

To bound the circumradius from above, we can use a similar argument as in the case $1 < p < n$. Since $\Theta(S_{K,p}) > 0$, defining once again the quantity $\delta = \frac{1}{2}\Theta(S_{K,p})$, and taking a direction $u_0$ such that $R_K u_0 \in \partial K$, we have 
\begin{align*}%\label{eq.upper-R_k-per_p>n}
    \delta 
    & \le 
    \int_{
        \{x \in \mathbb{S}^{n-1} \colon |\langle x,u_0\rangle| \ge \delta\}
    } 
    |\langle x, u_0\rangle| \, 
    \dd S_{K,p}(x) \cr 
    & = 
    \int_{
        \{x \in \mathbb{S}^{n-1} \colon |\langle x,u_0\rangle| \ge \delta\}
    } |\langle x,u_0\rangle| \, h_{K}^{1-p} \, 
    \dd S_K(x) \cr 
    & \le \delta^{1-p} R_K^{1-p} S_K(\{x \in \mathbb{S}^{n-1} \colon |\langle x,u_0\rangle| \ge \delta\}) \cr
    & \le \sigma_{n-1} \delta^{1-p} R_K^{n-p}.
\end{align*}
We hence obtain the desired upper bound on $R_K$
\begin{equation} 
R_K \le C_{n,p} \Theta(S_{K,p})^{-\frac{p}{p-n}},  
\end{equation}
with an explicit constant $C_{n,p} \eqdef \sigma_{n-1}^{\frac{1}{p-n}}$. 
    
    To finish the proof, we use a simplified argument to the one in the case $1 < p < n$. By taking the John ellipsoid  $E$ of $K$, with $a_1 \le \cdots \le a_n$ being its half axes, then 
    \[
        c_n \cdot r_K \cdot R_K^{n-1} 
        \ge 
        c_n \cdot a_1 \cdot \dots \cdot a_n \ge |K|
        \ge 
        \tilde{c}_n. 
    \]
    This, together with the upper bound we have on $R_K$, implies that $r_K$ admits a lower bound of the form 
    \[
        r_K \ge c_{n,p} \Theta(S_{K,p})^{p\frac{n-1}{p-n}},
    \]
    and the result follows.
\end{proof}

\section{Quantitative stability for $p=1$}\label{sec.case_p1}

Now we are ready to prove our main stability result for Minkowski's problem. Our proof relies on leveraging the stability of the Brunn-Minkowski inequality in terms of the Hausdorff distance, and later reinforce the stability obtained in this way with stability associated with the Fraenkel asymmetry. Our major reference for these tools are~\cite{diskant1972bounds,diskant1973stability,figalli2010stability}, whose results we review briefly in the sequel. 

One of the first approaches to reinforce the Brunn-Minkowski inequality quantitatively is due to Diskant~\cite{diskant1973stability}, where a control of the Hausdorff distance between two convex bodies can be obtained assuming bounds on the geometry of the convex bodies involved. Recalling that the Hausdorff distance of convex bodies can be written as the $L_\infty$ norm between the corresponding support functions 
\[
    \dH(E,F) = \norm{h_E - h_F}_{L_\infty(\mathbb{S}^{n-1})}. 
\]
We are then interested in controlling this quantity in terms of the deficit of the Brunn-Minkowski inequality, as we will only need this inequality with the convexity parameter given by $1/2$, we define the deficit of the Brunn-Minkowski inequality for two convex sets $E,F$ as 
\begin{equation}\label{eq.deficit_BM}
    \delta_{\BM}(E,F) 
    \eqdef 
    \frac{
        \left|\frac{1}{2}E + \frac{1}{2}F\right|^{\frac{1}{n}} 
    }{
        \frac{1}{2}|E|^{\frac{1}{n}} + 
        \frac{1}{2}|F|^{\frac{1}{n}}}
    - 1.
\end{equation}

It is shown in~\cite{diskant1973stability} that, given two convex bodies $E,F$ such that $r \le r_E, r_F$ and $R_E, R_F \le R$, there is a constant $C$ depending on $r,R,n$ such that 
\begin{equation}\label{eq.diskant_isoperimetric}
    \inf_{x_0 \in \mathbb{R}^n} {\dH(E, x_0 + \lambda F)}^n \le C\delta_{\BM}(E,F), 
\end{equation}
where $\lambda = \left(\frac{|E|}{|F|}\right)^{\frac{1}{n}}$. 

The optimal quantitative version of this inequality was then proposed in~\cite{figalli2010stability}, w.r.t.~the Fraenkel asymmetry. This result can then be viewed as a strong concavity of the volume functional with respect to Minkowski addition. However, it can be expected that this strong concavity is not uniform among all convex bodies and should be weaker for convex bodies with very different volumes. In other words, the stability constant depends on the volume ratio between $E$ and $F$ as
\[
    \sigma_{E,F} 
    \eqdef 
    \max 
    \left\{
      \frac{|E|}{|F|},\frac{|F|}{|E|}  
    \right\}. 
\]
It is shown in~\cite{figalli2010stability}, that there exists a constant $C_n$ depending only on the dimension for which we have that
\begin{equation}\label{eq.quantitative_BM}
    {\alpha(E,F)}^2 \le C_n\,\sigma^{\frac{1}{n}}_{E,F}\,\delta_{\BM}(E,F). 
\end{equation}
which controls then the stronger Fraenkel asymmetry between $E$ and $F$
\begin{equation}\label{eq.quantitative_isoperimetric}
    \alpha(E,F) 
    \eqdef 
    \inf\left\{
        \frac{|E \Delta (x_0 + rF)|}{|E|} : \
        x_0 \in \mathbb{R}^n, \ r^n|F| = |E|    
    \right\}. 
\end{equation}

\subsection{Proof of Theorem~\ref{thm.quantitative_stability_p=1} %via variational stability
}\label{sec.p=1_variational}

As mentioned above, our second approach is to exploit variational formulation that describes Minkowski bodies to derive stability estimates. A classical variational problem describing them is the following 
\begin{equation}\label{eq.classical_var_Minko}
    \lambda_\mu \eqdef 
    \inf_{K \text{ centered }} 
    F_\mu(K) 
    \eqdef 
    \frac{
        \displaystyle
        \int_{\mathbb{S}^{n-1}} h_K \dd \mu
    }{|K|^{\frac{1}{n}}}.
\end{equation}
Notice that the energy $F_\mu$ is invariant with respect to dilations. Therefore, we can show with Wulff's inequality that the only minimizers of $F_\mu$ are dilations of Minkowski's bodies associated with $\mu$. 

\begin{lemma}\label{lemma.variational_problem}
    The following assertions hold: 
    \begin{enumerate}
        \item Let $\mu \in \mathscr{P}(\mathbb{S}^{n-1})$ such that $\Theta(\mu) > 0$, then the quantity $K \mapsto F_\mu(K)$ is minimized by any dilation of the Minkowski body $E_\mu$ associated with $\mu$. As a result we have 
        \[
            \lambda_\mu = n {|E_\mu|}^{\frac{n-1}{n}}, 
        \]
        where $E_\mu$ is the unique centered body such that $S_{E_\mu} = \mu$.
        \item For $\mu,\nu$ such that $\Theta(\mu),\Theta(\nu) > 0$ it holds that 
        \[
            |\lambda_\mu - \lambda_\nu| 
            \le 
            C_{\mu,\nu} \dc(\mu,\nu),
        \]
        where $\dc$ corresponds to the convex dual distance, defined in~\eqref{eq.convex_dual_distance}, and $C_{\mu,\nu} =  \max\left\{\frac{R(E_\mu)}{|E_\mu|^{\frac{1}{n}}}, \frac{R(E_\nu)}{|E_\nu|^{\frac{1}{n}}}  \right\}$. 
    \end{enumerate}
\end{lemma}
\begin{proof}
    To prove item $(1)$, first notice that for any centered convex body $L$, by the anisotropic isoperimetric inequality, with anisotropy given by $h_L$, we have 
    \[
        \int_{\mathbb{S}^{n-1}} h_L \dd \mu 
        = 
        \int_{\mathbb{S}^{n-1}} h_L \dd S_{E_\mu} 
        = 
        \int_{\partial E_\mu} h_L(\nu_{E_\mu}) \dd \mathcal{H}^{n-1} 
        \ge 
        n{|L|}^{\frac{1}{n}} {|E_\mu|}^{\frac{n-1}{n}}. 
    \]
    It then follows that $F_\mu(L) \ge n|E_\mu|^{\frac{n-1}{n}}$ for all centered convex bodies $L$. Choosing $L = E_\mu$, from the identity $\int h_E \dd S_E = n|E|$ valid for any centered convex body, we see that $E_\mu$ attains the lower bound and so must be optimal for $\lambda_\mu$. Since this energy is invariant with respect to dilations, this conclusion holds for any $K = \lambda E_\mu$ and $\lambda > 0$. 
    
    To prove item (2), consider two measures $\mu,\nu$ and use $F_\mu(E_\nu)$ to bound $\lambda_\mu$ from above. Since the support function $h_{E_\nu}$ is Lipchitz continuous with constant given by the circumradius of $E_\nu$ we have 
    \[
        \lambda_\mu - \lambda_\nu 
        \le 
        \frac{1}{|E_\nu|^{\frac{1}{n}}} 
        \int_{\mathbb{S}^{n-1}} h_{E_\nu} \dd (\mu - \nu)
        \le 
        \frac{R(E_\nu)}{|E_\nu|^{\frac{1}{n}}} \dc(\mu,\nu). 
    \]
    Changing the roles of $\mu$ and $\nu$, the result follows. 
\end{proof}

Although fundamental, it is less clear how to exploit the stability of Brunn-Minkowski's inequality and the variational problem~\eqref{eq.classical_var_Minko} to obtain the desired stability estimates. A variant of a problem recently proposed by Carlier in~\cite{carlier2004theorem}, consists of exploring the equivalence between a convex body and its support function. Indeed, obtainning this function is clear, but given $\varphi \in \mathscr{C}_b(\mathbb{S}^{d-1})$, we can construct its associated convex body as 
\begin{equation}
    E_\varphi 
    \eqdef \left\{
        x \in \mathbb{R}^n : x\cdot \nu \le \varphi(\nu) 
        \text{ } \nu \in \mathbb{S}^{n-1}
    \right\}, 
    \text{ so that } 
    h_{E_\varphi} = \varphi.
\end{equation}
We could then consider an energy of the form $\displaystyle \varphi \mapsto |E_\varphi|^{\frac{1}{n}} - \int_{\mathbb{S}^{n-1}} \varphi \dd \mu$, which has the advantage of being convave due to the Brunn-Minkowski inequality, and we can hope to extract information from maximizers from~\eqref{eq.quantitative_BM}. 

The only issue is that this energy scales linearly with dilation, so we multiply it by a suitable constant to fix a natural scale of the energy. This is done by considering
\begin{equation}\label{eq.minko_energy_fixed_scale}
    \sup_{\varphi \in \mathscr{C}_b(\mathbb{S}^{n-1})} 
    \mathscr{F}_\mu(\varphi) 
    \eqdef 
    \lambda_\mu
    {|E_\varphi|}^{\frac{1}{n}}
    -
    \int_{\mathbb{S}^{n-1}} \varphi \dd \mu. 
\end{equation}
With this constant, the scale is fixed for the maximization of $\mathscr{F}_\mu$, and it selects precisely the Minkowski bodies, modulo translations, associated with $\mu$. 

\begin{lemma}\label{lemma.maxi_Fmu}
    A function $\varphi$ is maximizer of $\mathscr{F}_\mu$ if and only if $E_{\varphi}$ is homotethic to the Minkowski body $E_\mu$ associated with $\mu$. 
\end{lemma}
\begin{proof}
    Let $\varphi_\mu$ be the support function associated with $E_\mu$ the centered Minkowski body of $\mu$. By the equality case of the anisotropic isoperimetric inequality from above, for any support function $\varphi$ that is not equivalent to $\varphi_\mu$, for a translation and a dilation, we have 
    \begin{align*}
        \mathscr{F}_\mu(\varphi)
        &= 
        \lambda_\mu {|E_\varphi|}^{\frac{1}{n}}
        - 
        \int_{\mathbb{S}^{n-1}} \varphi \dd \mu 
        <
        \frac{\displaystyle
            \int_{\mathbb{S}^{n-1}} \varphi \dd \mu
        }{{|E_\varphi|}^{\frac{1}{n}}} {|E_\varphi|}^{\frac{1}{n}} 
        - 
        \int_{\mathbb{S}^{n-1}} \varphi \dd \mu = 0. 
    \end{align*}
    In addition, we have $\mathscr{F}_{\mu}(\varphi_\mu) = 0$, from the fact that the centered Minkowski body $E_\mu$ is optimal for $\lambda_\mu$. 
\end{proof}

\begin{remark}\label{remark.lojasiewicz}
    From the previous Lemma we can indeed appreciate the thesis of Theorem~\ref{thm.quantitative_stability_p=1} as a Łojasiewicz inequality in the sense of~\eqref{eq.lojasiewicz_inequalities}. Indeed, the equivalence class of homotheties that is hard encoded into the Fraenkel asymmetry is precisely the set of maximizers of the functionals $\mathscr{F}_\mu$. 
\end{remark}

The energy~\eqref{eq.minko_energy_fixed_scale} is convenient since the first variation of the volume functional can be explicitly computed with the surface measure as discussed above~\eqref{eq.first_variation_volume}. Indeed, for any $\varphi, v \in \mathscr{C}_b(\mathbb{S}^{n-1})$ and defining $\varphi_t \eqdef (1-t)\varphi + t v$ we have
\begin{equation}
    \lim_{t \to 0} \frac{|E_{\varphi_t}| - |E_{\varphi}|}{t} 
    = 
    \int_{\partial E_{\varphi}} v(\nu_{E_{\varphi}}) \dd \mathcal{H}^{n-1}
    =
    \int_{\mathbb{S}^{n-1}} v \dd S_{E_{\varphi}},
\end{equation}

As a result, the functional $\varphi \mapsto \mathscr{F}_\mu(\varphi)$ is concave and its first variation evaluated at a function $\varphi$ is the Radon measure given by 
\begin{equation}\label{eq.1st_order}
    \nabla \mathscr{F}_\mu[\varphi] 
    = 
    \lambda_\mu\,
    \frac{ S_{E_\varphi}}{n{|E_\varphi|}^{\frac{n-1}{n}}}
    -
    \mu. 
\end{equation}
As a result, the first order optimality condition for~\eqref{eq.minko_energy_fixed_scale} become $\nabla \mathscr{F}_\mu[\varphi] \equiv 0$, in the sense of measures, which is equivalent to being a Minkowski body, as we have already proven. 

At this point we have all the elements to obtain the desired stability result. 

\begin{proof}[Proof of Theorem~\ref{thm.quantitative_stability_p=1}]
    Given $\mu,\nu \in \mathscr{P}(\mathbb{S}^{n-1})$ such that $\Theta(\mu),\Theta(\nu) \ge \vartheta > 0$, the strict positivity of the functional $\Theta$ implies the existence of centered Minkowski bodies $E_\mu,E_\nu$. Let $\varphi_\mu, \varphi_\nu$ be their corresponding support functions.  

    Now, recall that from Lemma~\ref{lemma.variational_problem} we know that both support functions $\varphi_\mu, \varphi_\nu$ maximize the functionals $\mathscr{F}_\mu$ and $\mathscr{F}_\nu$, respectively. Hence, we define $\varphi_{1/2} \eqdef \frac{1}{2}(\varphi_\mu + \varphi_\nu)$, and consider the following quantity
    \begin{equation}\label{eq.def_delta}
        \begin{aligned}
            0 
            &\le \delta \eqdef 
            \mathscr{F}_\mu(\varphi_{1/2}) 
            - 
            \frac{1}{2}\left(
                \mathscr{F}_\mu(\varphi_{\mu}) + \mathscr{F}_\mu(\varphi_{\nu})
            \right)\\ 
            &= 
            {\left|\frac{1}{2}E_\mu + \frac{1}{2}E_\nu\right|}^{\frac{1}{n}} 
            - 
            \frac{1}{2}\left(
                |E_\mu|^{\frac{1}{n}} + |E_\nu|^{\frac{1}{n}}
            \right)\\ 
            &= 
            \frac{1}{2}\left(
                |E_\mu|^{\frac{1}{n}} + |E_\nu|^{\frac{1}{n}}
            \right) 
            \delta_{\BM}(E_\mu,E_\nu).
        \end{aligned}
    \end{equation}

    The quantity $\delta$ above is non-negative thanks to the concavity of $\varphi \mapsto \mathscr{F}_\mu(\varphi)$. So, using the fact that $\varphi_{\mu}$ is a maximizer and the aforementioned concavity, it follows that
    \begin{align*}
        \delta &\le \frac{1}{2}
        \left(
            \mathscr{F}_\mu(\varphi_\mu) 
            - 
            \mathscr{F}_\mu(\varphi_\mu) 
        \right)
        \le 
        \frac{1}{2}
        \inner{\nabla \mathscr{F}_\mu(\varphi_\nu), \varphi_\mu - \varphi_\nu}\\
        &= 
        \frac{1}{2}
        \inner{
            \frac{\lambda_\mu}{\lambda_\nu}S_{E_\nu} - \mu, \varphi_\mu - \varphi_\nu}
        = 
        \frac{1}{2\lambda_\nu}
        \inner{\lambda_\mu \nu - \lambda_\nu \mu, \varphi_\mu - \varphi_\nu}\\ 
        &= 
        \frac{\lambda_\mu}{2\lambda_\nu} \inner{\nu - \mu, \varphi_\mu - \varphi_\nu}
        + 
        \frac{\lambda_\mu - \lambda_\nu}{2\lambda_\nu} \inner{\mu, \varphi_\mu - \varphi_\nu}
        \\
        &\le 
        \frac{\lambda_\mu}{2\lambda_\nu} \norm{\varphi_\mu - \varphi_\nu}_\infty 
        \int_{\mathbb{S}^{n-1}} 1 \dd (\nu - \mu) 
        + 
        \frac{\lambda_\nu}{2}\mu(\mathbb{S}^{n-1})\norm{\varphi_\mu - \varphi_\nu}_\infty (\lambda_\nu - \lambda_\mu) \\ 
        &\le  
        C_{n,\vartheta}\, \dc(\mu,\nu)
    \end{align*}
    where we have used the quantitative stability of the constant $\mu \mapsto \lambda_\mu$ with respect to $\dc(\mu,\nu)$, proved in Lemma~\ref{lemma.variational_problem}, combined with the uniform estimates on the inradius and circumradius from~\ref{lemma.rR}, so that $\varphi_\mu$ and $\varphi_\nu$ are uniformly bounded by a constant depending only on $n$ and $\vartheta$. Combining the above estimate with~\eqref{eq.def_delta} and the quantitative stability of the Brunn-Minkowski inequality~\eqref{eq.quantitative_BM}, we conclude that there exists a constant $C_{\vartheta,n}$ such that 
    \[
        \alpha(E_\mu, E_\nu)^2 \le C_{n,\vartheta}\, {\dc(\mu,\nu)},
    \]
    for any pair of measures $\mu,\nu$. 

    As a result, we can obtain a strengthened version, since for $\mu,\nu$ sufficiently close, for instance in the $\dc$ distance or even in a stronger one such as the Wasserstein distances in optimal transport, we obtain that $E_\mu$ and $E_\nu$ are close enough to apply Diskant's Hausdorff based quantitative stability. However, since the convex bodies $E_\mu,E_\nu$ have uniformly bounded diameter and inradius, up to using a bigger stability constant this proximity condition on $\mu$ and $\nu$ can be ignored.
    
    Recalling that for any convex body $E$ and any displacement $x_0 \in \mathbb{R}^n$ it holds that 
    \[
        h_{x_0+E}(\theta) 
        = 
        \sup_{z \in x_0 + E} \inner{\theta,z} 
        = 
        \inner{x_0,\theta} + h_E(\theta), 
    \]
    and since $\varphi_{\mu} = h_{E_\mu}$, the estimate above and the fact that both measures $\mu$ and $\nu$ are centered at the origin gives that 
   \begin{align*}
        \delta \le&
        \frac{\lambda_\mu}{2\lambda_\nu} 
        \inner{\nu - \mu, h_{E_\mu}- h_{E_\nu + x_0}}
        + 
        \frac{\lambda_\mu - \lambda_\nu}{2\lambda_\nu} 
        \inner{\mu,  h_{E_\mu}- h_{E_\nu + x_0}} \\
        &+ 
        \frac{\lambda_\mu}{2\lambda_\nu} 
        \int_{\mathbb{S}^{n-1}}
        \inner{\theta, x_0}\dd(\nu - \mu)(\theta)
        +
        \frac{\lambda_\mu - \lambda_\nu}{2\lambda_\nu} 
        \int_{\mathbb{S}^{n-1}}
        \inner{\theta,x_0} \dd \mu (\theta)\\ 
        &\le 
        C_{n,\vartheta} \dc(\mu,\nu) \dH(E_\mu, x_0 + E_\nu). 
   \end{align*}

   We can then summarize the information that we have so far in the following inequalities
   \begin{align}
        \label{eq.estimate_deltaBM}
        \delta_{\BM}(E_\mu,E_\nu) 
        &\le 
        C_1 \dc(\mu,\nu) \inf_{x_0 \in \mathbb{R}^n} \dH(E_\mu, x_0 + E_\nu)\\
        \label{eq.BM_AMP}
        \alpha(E_\mu,E_\nu)^2
        &\le 
        C_2 \delta_{\BM}(E_\mu,E_\nu) \\
        \label{eq.BM_Diskant}
        \inf_{x_0 \in \mathbb{R}^n} 
        \dH(E_\mu, x_0 + \lambda E_\nu)^n
        &\le 
        C_3 \delta_{\BM}(E_\mu,E_\nu),
   \end{align}  
   where $\lambda = \left(\frac{|E_\mu|}{|E_\nu|}\right)^{\frac{1}{n}}$.
   
   Estimate~\eqref{eq.estimate_deltaBM} is obtained by taking the infimum among all $x_0$ in the previous argument,~\eqref{eq.BM_AMP} follows from the Fraenkel asymmetry based quantitative Brunn-Minkowski inequality, while~\eqref{eq.BM_Diskant} corresponds to Diskant's version based on the Hausdorff distance.

   Using the volume stability bounds from Lemma~\ref{lemma.variational_problem}, we can remove the dilation $\lambda$ in~\eqref{eq.BM_Diskant} at the cost of a bigger stability constant. Indeed, as 
   \[
        \frac{\lambda_\mu}{\lambda_\nu} = \lambda^{n-1},
   \]
   we have that $\lambda$ is uniformly bounded from above by a constant depending only on $n$ and $\vartheta$, and we can assume $\lambda \ge 1$ without loss of generality by taking $\mu,\nu$ such that $|E_\mu| \ge |E_\nu|$. In addition, from Lemma~\ref{lemma.variational_problem}, there is a constant depending on $n$ and $\vartheta$ such that 
   \[
      \lambda \le 
      {\left(1 + C\dc(\mu,\nu)\right)}^{\frac{1}{n-1}} 
      \le
      1 + \frac{2C}{n-1}\dc(\mu,\nu). 
   \]
   
   Combining~\eqref{eq.estimate_deltaBM} and~\eqref{eq.BM_Diskant} we obtain the following Hausdorff stability for Minkowski bodies 
   \begin{align*}
        \inf_{x_0 \in \mathbb{R}^n} \dH(E_\mu, x_0 + \lambda E_\nu)^n
        &\le 
        C_1 \dc(\mu,\nu)
        \left( 
            \dH(E_\nu, \lambda E_\nu)    
            + 
            \inf_{x_0 \in \mathbb{R}^n} 
            \dH(E_\mu, x_0 + \lambda E_\nu)
        \right)\\ 
        &\le 
        C \dc(\mu,\nu)^2 + 
        C' \dc(\mu,\nu) 
        \inf_{x_0 \in \mathbb{R}^n} \dH(E_\mu, x_0 + \lambda E_\nu).
   \end{align*}
   Noticing that $\dc(\mu,\nu)$ is bounded by a constant depending only on $n$, we can absorb the first term in the right-hand side into the second one, up to increasing the constant. Therefore, we can simplify the exponent on the left-hand side, and using a similar argument to remove the dilation $\lambda$ we obtain the following Hausdorff stability for Minkowski bodies
   \begin{equation}\label{eq.hausdorff_estimate}
       \inf_{x_0 \in \mathbb{R}^n} \dH(E_\mu, x_0 + E_\nu)
        \le 
        C_1 \dc(\mu,\nu)^{\frac{1}{n-1}} 
   \end{equation}
   We can then combine this with~\eqref{eq.estimate_deltaBM} and~\eqref{eq.BM_AMP} to obtain 
   \begin{equation}\label{eq.fraenkel_estimate}
        \alpha(E_\mu,E_\nu)^2
        \le 
        C {\dc(\mu,\nu)}^{1 + \frac{1}{n-1}}.
   \end{equation}
\end{proof}

We now discuss the sharpness of the previous estimates. This will be done with the following two examples.

\begin{example}\label{exemple.sharp_hausdorff}
    Let $Q_0 = {[-1,1]}^n$ be the centered cube of side $2$ in $\mathbb{R}^n$. Our sequence of convex bodies will be obtained by slicing this cube by the hyperplane $\{x \in \mathbb{R}^n : x_1 + \cdots + x_n = n - t\}$: 
    \[
        Q_t \eqdef 
        Q_0 \cap \left\{
            x \in \mathbb{R}^n : x_1 + \cdots + x_n \le n - t
        \right\}
    \]
    As a result, we get that $\dH(Q_0, Q_t) = O(t)$. On the other hand, the new facet introduced by this slicing is given by a simplex $\Delta_t$ of area proportional to $O(t^{n-1})$. As a result, we get for all $1$-Lipchitz functions $f$ that 
    \[
        \left|
            \int_{\mathbb{S}^{n-1}}
            f \dd(S_{Q_t} - S_{Q_0})
        \right|
        \le C \mathcal{H}^{n-1}(\Delta_t) \le C t^{n-1}. 
    \]
    Since the quantities $\Theta(Q_t)$ remain uniformly bounded from below as $t \to 0$, this example shows that any estimate of the form 
    \[
        \inf_{x_0 \in \mathbb{R}^n} 
        \dH(E, F + x_0) \le C 
        \dc(S_E, S_F)^\gamma
    \]
    cannot hold with an exponent $\gamma$ bigger than $1/(n-1)$. Since an estimate was obtained with this upper bound in~\eqref{eq.hausdorff_estimate}, this is the sharp exponent. 
\end{example}

\begin{example}\label{exemple.sharp_fraenkel}
    Consider once again the cube $Q_0 = {[-1,1]}^n$. And now, instead of slicing the cube, we enlarge one of its sides defining 
    \[
        \bar Q_t \eqdef 
       {[-1,1]}^{n-1}\times [-1-t,1 + t]. 
    \]
    In this case, the volume difference between $Q_t$ and $Q_0$ behaves linearly in $t$ and we get $\alpha(Q_0, \bar Q_t) = O(t)$. On the other hand, the surface area also only grows linearly in $t$ and we get that $\dc(S_{\bar Q_t}, S_{Q_0}) = O(t)$ as well.

    Once again the quantities $\Theta(\bar Q_t)$ remain uniformly bounded from below as $t \to 0$, so that this example indicates a Lipschitz behaviour
    \[
        \alpha(E, F) \le C 
        \dc(S_E, S_F).
    \]
    In dimension $n = 2$ this behavior is obtained in~\eqref{eq.fraenkel_estimate}, showing that this is the sharp exponent on the plane. We conjecture that the Lipschitz behavior should also hold in $\mathbb{R}^n$.  
\end{example}

\section{Quantitative stability for $1 < p \neq n$}
\label{secp>1}

In order to adapt the previous arguments for the $p > 1$ case, the central object becomes Firey's $L_p$ Brunn-Minkowski theory. As already discussed in Section~\ref{sec.technical_estimates_p>1}, the $L_p$ theory is strongly based on the equivalence between a convex body $E$ and its support function $h_E$, we recall that the $L_p$-sum of convex bodies is defined via a non-linear combination of their respective support functions as
\begin{equation}
    h_{t\cdot E +_p s\cdot F} \eqdef 
    {\left(
        t h_E^p + s h_F^p
    \right)}^{1/p}.
\end{equation}
The convex body that defines the corresponding $L_p$ Minkowski sum is then given by 
\begin{equation}
    t\cdot E +_p s\cdot F
    \eqdef 
    \left\{
        x \in \mathbb{R}^n: 
        \inner{x,u}
        \le 
        {\left( t h_E^p(u) + s h_F^p(u) \right)}^{\frac{1}{p}}, \text{ for all }
        u \in \mathbb{S}^{n-1}
    \right\}
\end{equation}

As in the classical $p=1$ case, a suitable power of the volume functional is concave with respect to the $L_p$ Minkowski sum. Namely, for any pair of convex bodies $E,F$ which contain the origin it holds that 
\begin{equation}\label{eq.LpBM_}
    {|t\cdot E +_p s \cdot F|}^{p/n} \ge 
    t{|E|}^{p/n} + (1-t){|F|}^{p/n}. 
\end{equation}

This theory has already given us the analogous control to the case $p = 1$ of the geometry of a convex body $E$, whenever the quantity $\Theta(S_{E,p})>0$, is strictly positive. Since this was the major technical result that allowed our uniform control in the proof of Theorem~\ref{thm.quantitative_stability_p=1}, this encourages the extension of both arguments -- the variational one, and via isoperimetric rigidity -- to the $L_p$ case.

To the best of the authors' knowledge, there is no ``off-the-shelf'' quantitative stability results for the $L_p$ isoperimetric and Brunn-Minkowski inequalities, in analogies to the classical couterparts~\eqref{eq.quantitative_isoperimetric} and~\eqref{eq.quantitative_BM}. For this reason we dedicate the next paragraph to the derivation of these stronger versions in the $L_p$ case.

\subsection{Quantitative versions of $L_p$ Brunn-Minkowski and isoperimetric inequalities}\label{sec.quantitativeLpBM_iso}

Firey's $L_p$ Brunn-Minkowski inequalities can be derived from the classical $p=1$ case, see for instance~\cite[Thm.~7.6.3]{boroczky2025isoperimetric} or~\cite{firey1962p}. On the other hand, as was already remarked in Section~\ref{sec.technical_estimates_p>1}, the $L_p$ isoperimetric inequalities can be derived from the $L_p$-BM theory. Therefore, in this section we simply use the stronger quantitative versions from the $p=1$ case in the proof of these results. 

\begin{lemma}\label{lemma.quantitative_LpBM}
    Given $p > 1$, let $E,F \subset \mathbb{R}^n$ be convex bodies containing the origin. Then there is some $\tau \in (0,1/2]$ such that $\tau \le t \le 1 - \tau$, and a dimensional constant $\theta_n \eqdef cn^{-4}{(\log n)}^{-1}$, with $c \in (0,1)$ such that
    \begin{equation}
        |t\cdot E +_p (1-t) \cdot F| \ge 
        |E|^t|F|^{1-t} 
        \left(
            1 + \theta_n \tau {\alpha(E,F)}^2
        \right).
    \end{equation}
    In addition, for all $t,s > 0$ we have 
    \begin{equation}
        {|t\cdot E +_p s\cdot F|}^{p/n} 
        \ge 
        \left(
            t {|E|}^{p/n} + s {|F|}^{p/n}
        \right)
        {\left(
            1 + \frac{\theta_n}{
                \sigma^{p/n}_{E,F}
            } {\alpha(E,F)}^2
        \right)}^{p/n},
    \end{equation}
    where $\sigma_{E,F} \eqdef \max \left\{
        \frac{|E|}{|F|},\frac{|F|}{|E|}
    \right\}$. 

    Using Diskant's version of the Brunn-Minkowski inequality, given $0 <r <R$, there is a constant $C$ depending on $r,R,p,n$, and the convexity parameter $t$ such that for any pair of convex bodies $E,F$ such that $r\le r_E, r_F$ and $R_E,R_F \le R$, it holds that 
    \begin{equation}\label{eq.hausdorff_Lp_BM}
        {|t\cdot E +_p s\cdot F|}^{p/n} 
        \ge 
        \left(
            t {|E|}^{p/n} + s {|F|}^{p/n}
        \right)
        {\left(
            1 + C \inf_{x \in \mathbb{R}^n} 
        \dH(E, x + \lambda F)^n
        \right)}^{p/n},
    \end{equation}
    where $\lambda = \left(\frac{|E|}{|F|}\right)^{\frac{1}{n}}$.
\end{lemma}
\begin{proof}
    The proof is obtained by using the quantitative stability version of the Brunn-Minkowski inequality for convex sets into the proof of the $L_p$ Brunn-Minkowski inequality found for instance in~\cite{boroczky2025isoperimetric}. 

    For the first case, take $t \in (0,1)$, so that by convexity of the function $r \mapsto r^p$ for $p > 1$, we have that 
    \[
        h^p_{t\cdot E +_p (1-t)\cdot F} 
        = 
        t h_E^p + (1-t) h_F^p
        \ge 
        (t h_E + (1-t) h_F)^p. 
    \]
    So that 
    \[
        t E + (1-t) F \subseteq t\cdot E +_p (1-t)\cdot F. 
    \]
    Using the quantitative stability version of the classical Brunn-Minkowski inequality~\cite[Corollary~8.6.5]{boroczky2025isoperimetric}, this implies that 
    \[
        |t\cdot E +_p (1-t)\cdot F|
        \ge 
        |t E + (1-t) F|
        \ge 
        |E|^t|F|^{1-t}
        \left(
            1 + \theta_n \tau \alpha(E,F)^2
        \right),
    \]
    where $\tau \in (0,1/2]$ is a such that $\tau \le t \le 1 - \tau$, and $\theta_n \eqdef cn^{-4}{(\log n)}^{-1}$.  

    Now let us derive the second inequality set $t_0 \eqdef |E|^{\frac{1}{n}}$, $s_0 \eqdef |F|^{\frac{1}{n}}$, as well as the dilations 
    \[
        E_0 \eqdef {t_0}^{-1}E, \quad 
        F_0 \eqdef {s_0}^{-1}F,
    \]
    in such a way that $|E_0| = |F_0| = 1$. Therefore, considering 
    \[
        \lambda \eqdef \frac{s\,s_0^p}{t\,t_0^p + s\,s_0^p} \in (0,1),
    \]
    and assuming without loss of generality that $\frac{s\ s_0^p}{t\,t_0^p + s\,s_0^p} < \frac{t\,t_0^p}{t\,t_0^p + s\,s_0^p}$, we have that 
    \begin{align*}
        h_{t\cdot E +_p s\cdot F}
        &= 
        {\left(
            t\,h_{E}^p + s\,h_{F}^p
        \right)}^{1/p}  \\ 
        &= 
        {\left(
          t\,t_0^p + s\,s_0^p
        \right)}^{1/p} 
        {\left(
            (1-\lambda) h_{E_0}^p +
            \lambda h_{F_0}^p
        \right)}^{1/p}
    \end{align*}

    As a result, using the previous quantitative stability version of the Brunn-Minkowski inequality with the renormalized sets $E_0, F_0$, we get that 
    \begin{align*}
        {(t\,t_0^p + s\,s_0^p)}^{\frac{n}{p}}
        |t\cdot E +_p s\cdot F|
        &= 
        |(1-\lambda)\cdot E_0 +_p \lambda\cdot F_0|\\ 
        &\ge \left(
            1 + \theta_n \tau {\alpha(E_0, F_0)}^2
        \right),
    \end{align*}
    for $\tau$ such that 
    \[
        \lambda \le \tau \le 1-\lambda. 
    \]
    Recalling the definition of $t_0$ and $s_0$ above, and dilating $E_0,F_0$ back to $E,F$ it clearly holds that $\alpha(E_0,F_0) = \alpha(E,F)$ and we get that 
    \[
        {|t\cdot E +_p s\cdot F|}^{p/n} 
        \ge 
        \left(
            t {|E|}^{p/n} + s {|F|}^{p/n}
        \right)
        {\left(
            1 + \frac{\theta_n}{
                \sigma_{E,F}^{p/n}
            } {\alpha(E, F)}^2
        \right)}^{p/n}
    \]

    The estimate based on the Hausdorff distance~\eqref{eq.hausdorff_Lp_BM} is obtained by using the same argument, but instead of using the quantitative stability version of the Brunn-Minkowski inequality based on the Fraenkel asymmetry, we use the one due to Diskant~\cite{diskant1973stability} discussed above. 
\end{proof}
This proof is most likely not optimal due to the exponent $p/n$ that seems inevitable with this approach. It is nevertheless sufficient to conclude with the quantitative stability result in the $L_p$ case.

On the other hand, leveraging the quantitative anisotropic isoperimetric inequality, we can sharpen the inequality on the $L_p$ mixed volume
\begin{equation}\label{eq.Lp_Minkowski_inequality}
    V_p(E,F) \ge 
    |F|^{\frac{p}{n}}{|E|}^{\frac{n-p}{n}},
    \text{ equivalent to }
    0\le 
    \delta_{\mathrm{ISO},p}(E,F) 
    \eqdef 
    \frac{V_p(E,F)}{
        |F|^{\frac{p}{n}}{|E|}^{\frac{n-p}{n}}
    } - 1,
\end{equation}
where we have equality if and only if $E$ and $F$ are dilates. Although we shall not require this version, we believe it to be of independent interest, and it is a direct consequence of the quantitative anisotropic isoperimetric inequality using the same ideas from the quantitative $L_p$ Brunn-Minkowski inequality above. 

% We are interested in using the deficits 
% \begin{align*}
%     \delta_{\text{BM},p}(E,F) 
%     &\eqdef 
%     \frac{\left|
%         \frac{1}{2}E 
%         +_p 
%         \frac{1}{2}F
%     \right|^{\frac{p}{n}} }{
%         \frac{1}{2} |E|^{\frac{p}{n}} 
%         + 
%         \frac{1}{2} |F|^{\frac{p}{n}}
%     }-1, \ 
%     \delta_{\mathrm{ISO},p}(E,F) 
%     \eqdef 
%     \frac{P_{F,p}(E)}{
%         n|F|^{\frac{p}{n}}{|E|}^{\frac{n-p}{n}}
%     } - 1,
% \end{align*}
% to control the Fraenkel asymmetry directly. 

% In the regime $1 < p < n$, we derive a quantitative stability of inequality~\eqref{eq.Lp_Minkowski_inequality}, while in the regime $n < p$ we obtain a stronger version to the one found in Lemma~\ref{lemma.quantitative_LpBM}. These regimes dictate the convexity or concavity of $t \mapsto t^{p/n}$, which is explored in the following Lemma.

\begin{lemma}\label{lemma.quantitative_Lp_isoperimetric}
    Given $p > 1$, there exists a constant $C_n > 0$, depending only on the dimension, such that for any convex bodies $E,F$ we have that
    \begin{equation}\label{eq.fraenkel_Lp_isoperimetry}
        C_n\, \alpha(E,F)^2 \le \delta_{\mathrm{ISO},p}(E,F).
    \end{equation}

    Using Diskant's version of the quantitative isoperimetric inequality, given $0 <r <R$, there is a constant $C$ depending on $r,R,p$ and $n$ such that for any pair of convex bodies $E,F$ such that $r\le r_E, r_F$ and $R_E,R_F \le R$, it holds that 
    \begin{equation}\label{eq.hausdorff_Lp_isoperimetry}
        \inf_{x \in \mathbb{R}^n} 
        \dH(E, x + \lambda F)^n\le C \delta_{\mathrm{ISO},p}(E,F), 
    \end{equation}
    where $\lambda = \left(\frac{|E|}{|F|}\right)^{\frac{1}{n}}$.
\end{lemma}
\begin{proof}
    Defining the volume measure $V_E \eqdef \frac{1}{n} h_E S_E$, we recall that $|E| = V_E(\mathbb{S}^{n-1})$. As a result, from Jensen's inequality we have that 
    \begin{align*}
        {V_p(E,F)}^{1/p} 
        &= 
        {|E|}^{1/p}
        {\left(
            \int_{\mathbb{S}^{n-1}} 
            \frac{h^p_F}{h^p_E}\,
            \frac{\dd V_E}{|E|}
        \right)}^{1/p}
        \ge 
         {|E|}^{1/p}
        \int_{\mathbb{S}^{n-1}} 
        \frac{h_F}{h_E}\,
        \frac{\dd V_E}{|E|}\\
        &= 
        |E|^{\frac{1}{p} - 1}
        \frac{1}{n}
        \int_{\mathbb{S}^{n-1}} 
        h_F\dd S_E 
        = 
        |E|^{\frac{1}{p} - 1}
        V_1(E,F)
        \\
        & \ge 
        |F|^{\frac{1}{n}} 
        |E|^{\frac{1}{p} - \frac{1}{n}} 
        \left(
            1 + C_n\alpha(E,F)^2
        \right),
    \end{align*}
    where in the last identity we have used the quantitative anisotropic isoperimetric inequality. This yields that
    \[
        V_p(E,F)
        \ge 
        |F|^{\frac{p}{n}} 
        |E|^{\frac{n-p}{n}} 
        \left(
            1 + C_n\alpha(E,F)^2
        \right)^p.
    \]
    As $p > 1$ and the quantity on the $p$-th power on the right-hand side is at least $1$, we have that 
    \[
        C_n\, \alpha(E,F)^2 \le 
        \delta_{\mathrm{ISO},p}(E,F),
    \]
    and the result follows.  

    By employing Diskant's version of the isoperimetric inequality~\eqref{eq.diskant_isoperimetric} in the above argument, we obtain~\eqref{eq.hausdorff_Lp_isoperimetry}. 
\end{proof}

% \subsection{Proof of Theorem~\ref{thm.quantitative_stability_p} via $L_p$ isoperimetric rigidity} 

\subsection{Proof of Theorem~\ref{thm.quantitative_stability_p} %via $L_p$ variational stability
}

We can also adapt the variational proof developed in Section~\ref{sec.p=1_variational} to prove Theorem~\ref{thm.quantitative_stability_p}. In this case, using the $L_p$ Brunn-Minkowski theory, the quantity $E \mapsto {|E|}^{p/d}$ is concave for the $L_p$ Minkowski sum. For the formulation in terms of functions $\varphi \in \mathscr{C}_b(\mathbb{S}^{n-1})$, we have that the quantity 
\[
    \varphi \mapsto {|E_\varphi|}^{\frac{p}{n}}
\]
is concave but for a non-linear type of variation, \textit{i.e.} $\varphi_t \mapsto {|E_{\varphi_t}|}^{\frac{p}{n}}$ is concave with $\varphi_t \eqdef {\left((1-t)\varphi_0^p + t \varphi_1^p\right)}^{1/p}$. 

As a result, if before the central idea was balancing a suitable power of the volume with a linear functional induced by the measure $\mu$, now we need to consider functionals which are linear w.r.t.~the type of variations that make $\varphi \mapsto {|E_\varphi|}^{\frac{p}{n}}$ concave. So in this framework the corresponding linear functional becomes
\[
    \varphi \mapsto \int_{\mathbb{S}^{n-1}} \varphi^p \dd \mu.
\]

Once again, the quantities $t \mapsto {|E_{t\varphi}|}^{\frac{p}{n}}$ and $t \mapsto \int_{\mathbb{S}^{n-1}} {(t\varphi)}^p \dd \mu$ scale both like $t^p$. So to maximize the corresponding energy to~\eqref{eq.minko_energy_fixed_scale}, we need to fix the scale with a suitable constant that will be defined as 
\begin{equation}
    \lambda_{\mu,p} 
    \eqdef 
    \inf_{K} 
    F_{\mu,p}(K) 
    \eqdef
    \frac{
        \displaystyle 
        \int_{\mathbb{S}^{n-1}} {h_K}^p \dd \mu
    }{
        |K|^{\frac{p}{n}}
    },
\end{equation}
where the infimum is taken among all convex bodies $K$ such that $o \in K$ and $S_{K,p}(\{h_K = 0\}) = 0$. So the energy whose maximization yields $L_p$ Minkowski bodies becomes
\begin{equation}\label{eq.prob_L_pMinkowski}
    \sup_{\varphi \in \mathscr{C}_b(\mathbb{S}^{n-1})} 
    \mathscr{F}_{\mu,p}(\varphi) 
    \eqdef 
    \lambda_{\mu,p}
    {|E_{\varphi}|}^{p/n}
    -
    \int_{\mathbb{S}^{n-1}} \varphi^p \dd \mu. 
\end{equation}
As in the $p = 1$, with this constant, the correct scale is fixed for the maximization of $\mathscr{F}_{\mu,p}$, and it selects precisely the Minkowski bodies associated with $\mu$. The analogous properties to the ones found in Lemma~\ref{lemma.variational_problem} are collected in the following result, whose proof is completely analogous to the $p=1$ case detailed in Section~\ref{sec.p=1_variational}. 

\begin{lemma}\label{lemma.variational_problem_Lp}
    For all $1\le p \neq n$, the following assertions hold: 
    \begin{enumerate}
        \item Let $\mu \in \mathscr{P}(\mathbb{S}^{n-1})$ such that $\Theta_+(\mu) > 0$, then the quantity $K \mapsto F(K)$ is minimized by any dilation of the $L_p$ Minkowski body $E_{\mu,p}$ associated with $\mu$. As a result we have 
        \[
            \lambda_{\mu,p} = \frac{n}{p} {|E_{\mu,p}|}^{\frac{n-p}{n}}, 
        \]
        where $E_{\mu,p}$ is the unique convex body such that $S_{E_{\mu,p} , p} = \mu$.
        \item A support function $\varphi$ maximizes the energy $\mathscr{F}_{\mu,p}$ if, and only if, it is a dilation of the support function $\varphi_{\mu,p}$ associated with the $L_p$ Minkowski body $E_{\mu,p}$ of $\mu$. 
        \item For $\mu,\nu$ such that $\Theta(\mu),\Theta(\nu) > 0$ it holds that 
        \[
            |\lambda_{\mu,p} - \lambda_{\nu,p}| 
            \le 
            C_{\mu,\nu} \dc(\mu,\nu),
        \]
        where $C_{\mu,\nu} =  \max\left\{\frac{R(E_{\mu,p})}{|E_{\mu,p}|^{p/n}}, \frac{R(E_{\nu,p})}{|E_{\nu,p}|^{p/n}}  \right\}$. 
    \end{enumerate}
\end{lemma}

The final ingredient is the generalization of Alexandrov's Lemma to the $L_p$ case, proven by Lutwak~\cite{lutwak1992brunn-minkowki-firey}, which gives us the first variation of the volume functional with variations performed with the $L_p$ Minkowski sum. This was already discussed in Section~\ref{sec.technical_estimates_p>1}, see~\eqref{eq.first_variation_volume_Lp}. In particular, it gives the first variation of $\mathscr{F}_{\mu,p}$ analogously to~\eqref{eq.1st_order}. For two functions $\varphi_0, \varphi_1$ and $\varphi_t \eqdef \left((1-t)\varphi_0^p + t \varphi_1^p\right)^{1/p}$, we have that 
\begin{equation}\label{eq.Alexandrov_Lp}
    \frac{\dd}{\dd t}\Big|_{t = 0} 
    \mathscr{F}_{\mu,p}(\varphi_t) 
    = 
    \inner{\nabla_p \mathscr{F}_{\mu,p}(\varphi_0), \varphi_1^p - \varphi_0^p}, 
\end{equation}
and $\nabla_p \mathscr{F}_{\mu,p}(\varphi_0)$ corresponds to the Radon measure given by 
\[
    \nabla_p \mathscr{F}_{\mu,p}(\varphi_0)
    = 
    \frac{\lambda_{\mu,p}}{\frac{n}{p}|E_{\varphi_0}|^{\frac{n-p}{n}}} S_{E_{\varphi_0},p} - \mu.
\]

These properties give us all the elements to adapt the variational proof of Theorem~\ref{thm.quantitative_stability_p=1} to the $L_p$ case.

\begin{proof}[Proof of Thm.~\ref{thm.quantitative_stability_p}]
    Consider $\mu,\nu \in \mathscr{P}(\mathbb{S}^{n-1})$ such that $\Theta_+(\mu),\Theta_+(\nu) \ge \vartheta > 0$. From the positivity of the functional $\Theta_+$, Theorem~\ref{thm.existence_Minko_Lp} asserts the existence of $L_p$ Minkowski bodies $E_{\mu,p},E_{\nu,p}$ associated with $\mu,\nu$. Let $\varphi_{\mu,p}$ and $\varphi_{\nu,p}$ be the corresponding support functions.

    As in the $p = 1$ case, we can use the both the strong Fraenkel asymmetry and Hausdorff distance rigidity from Lemma~\ref{lemma.quantitative_LpBM}, by leveraging the variational formulation discussed above. We know from Lemma~\ref{lemma.variational_problem_Lp} that $\varphi_{\mu,p}$ and $\varphi_{\nu,p}$ maximize the functionals $\mathscr{F}_{\mu,p}$ and $\mathscr{F}_{\nu,p}$, respectively. As in the $p=1$ case, we want to exploit the concavity of $\mathscr{F}_{\mu,p}$, but this time this concavity is expressed in terms of the following nonlinear variation 
    \[
        \varphi_{1/2} \eqdef {\left(\frac{1}{2}\varphi_{\mu,p}^p + \frac{1}{2}\varphi_{\nu,p}^p\right)}^{1/p},
    \] 
    which also has the advantage of preserving the portion of the energy coming from $\varphi \mapsto \int_{\mathbb{S}^{n-1}} \varphi^p \dd \mu$. 
    
    From the $L_p$ Brunn-Minkowski inequality, the following quantity is non-negative
    \begin{equation}\label{eq.def_delta}
        \begin{aligned}
            0 
            \le \delta 
            \eqdef&
            \mathscr{F}_{\mu,p}(\varphi_{1/2}) 
            - 
            \frac{1}{2}\left(
                \mathscr{F}_{\mu,p}(\varphi_{\mu}) + \mathscr{F}_{\mu,p}(\varphi_{\nu})
            \right)\\ 
            =& 
            {\left|\frac{1}{2}E_{\mu,p} + \frac{1}{2}E_{\nu,p}\right|}^{\frac{1}{n}} 
            - 
            \frac{1}{2}\left(
                |E_{\mu,p}|^{p/n} + |E_{\nu,p}|^{p/n}
            \right)\\
            &+ 
            \int_{\mathbb{S}^{n-1}}
            \underbrace{
            \left(
                    \varphi_{1/2}^p 
                    -\frac{1}{2}\varphi_{\mu,p}^p
                    -\frac{1}{2}\varphi_{\nu,p}^p
                \right)  
            }_{=0} \dd \mu \\ 
            =& 
            {\left|\frac{1}{2}E_{\mu,p} + \frac{1}{2}E_{\nu,p}\right|}^{p/n} 
            - 
            \frac{1}{2}\left(
                |E_{\mu,p}|^{p/n} + |E_{\nu,p}|^{p/n}
            \right).
        \end{aligned}
    \end{equation}

    Since the volumes $|E_{\mu,p}|,|E_{\nu,p}|$ are uniformly controlled from above and from below, up to using a worse constant in the quantitative $L_p$ Brunn-Minkowski inequality from~\ref{lemma.quantitative_LpBM} we have the bound
    \begin{equation}
        \begin{aligned}
            {\left|\frac{1}{2}E_{\mu,p} + \frac{1}{2}E_{\nu,p}\right|}^{p/n} 
            &\ge 
            \frac{1}{2}\left(
                |E_{\mu,p}|^{p/n} + |E_{\nu,p}|^{p/n}
            \right) 
            \left(
                1 + C'_{n,p,\vartheta}\ \alpha(E_{\mu,p},E_{\nu,p})^2
            \right)^{p/n}\\ 
            &\ge
            \frac{1}{2}\left(
                |E_{\mu,p}|^{p/n} + |E_{\nu,p}|^{p/n}
            \right) 
            \left(
                1 + \frac{p}{n}C'_{n,p,\vartheta}\  \alpha(E_{\mu,p},E_{\nu,p})^2
            \right)
        \end{aligned}
    \end{equation}
    As a result, Lemma~\ref{lemma.quantitative_LpBM} yields the following control 
    \begin{equation}\label{eq.control_Fraenkel_Lp}
        \inf_{x \in \mathbb{R}^n} 
        \dH(E_{\mu,p},x + E_{\nu,p})^n 
        \le C_{n,p,\vartheta}\ \delta
        \text{ and }
        \alpha(E_{\mu,p},E_{\nu,p})^2 \le C_{n,p,\vartheta}\ \delta. 
    \end{equation}

    As in the $p = 1$ case, to control the quantity $\delta$ above we use the concavity of the function $t \mapsto F(t) \eqdef \mathscr{F}_{\mu,p}(\varphi_{t})$, for $\varphi_t \eqdef {\left( (1-t)\varphi_{\mu,p}^p + t\varphi_{\nu,p}^p \right)}^{1/p}$, along with the optimality of $\mathscr{F}_{\mu,p}(\varphi_{1/2}) \le \mathscr{F}_{\mu,p}(\varphi_{\mu,p})$. With the analogous arguments from the $p=1$ case, we have that 
    \begin{align*}
        \delta 
        &= 
        \mathscr{F}_{\mu,p}(\varphi_{1/2}) 
        - 
        \frac{1}{2}
        \left(
            \mathscr{F}_{\mu,p}(\varphi_{\mu,p}) 
            + 
            \mathscr{F}_{\mu,p}(\varphi_{\mu,p}) 
        \right)
        \le \frac{1}{2}
        \left(
            \mathscr{F}_{\mu,p}(\varphi_{\mu,p}) 
            - 
            \mathscr{F}_{\mu,p}(\varphi_{\mu,p}) 
        \right)\\ 
        &\le 
        \frac{1}{2}F'(0) = 
        \frac{1}{2}
        \inner{\nabla_p \mathscr{F}_{\mu,p}(\varphi_{\nu,p}), \varphi_{\mu,p}^p - \varphi_{\nu,p}^p}\\
        &= 
        \frac{1}{2}
        \inner{
            \frac{\lambda_{\mu,p}}{\lambda_{\nu,p}}S_{E_{\nu,p}} - \mu, \varphi_{\mu,p}^p - \varphi_{\nu,p}^p}
        = \frac{1}{2\lambda_{\nu,p}}
        \inner{\lambda_{\mu,p} \nu - \lambda_{\nu,p} \mu, \varphi_{\mu,p}^p - \varphi_{\nu,p}^p}\\ 
        &\le 
        C_{n,p,\vartheta}
        \inner{\lambda_{\mu,p} \nu - \lambda_{\nu,p} \mu, \varphi_{\mu,p} - \varphi_{\nu,p}}. 
    \end{align*}

    Using the quantitative stability of the constant $\mu \mapsto \lambda_{\mu,p}$ with respect to $\dc(\mu,\nu)$, from Lemma~\ref{lemma.variational_problem_Lp}, this last quantity can then be bounded with the dual convex distance $\dc(\mu,\nu)$, giving the following estimates
    \begin{equation}%\label{eq.control_Fraenkel_Lp}
        \inf_{x \in \mathbb{R}^n} 
        \dH(E_{\mu,p},x + E_{\nu,p})^n 
        \le C_{n,p,\vartheta} \dc(\mu,\nu)
        \text{ and }
        \alpha(E_{\mu,p},E_{\nu,p})^2 \le C_{n,p,\vartheta}
        \dc(\mu,\nu).  
    \end{equation}

    In the particular case that both $\mu,\nu$ are centered at the origin, we can obtain a stronger estimate by means of Diskant's quantitative Brunn-Minkowski theory employed for $p=1$. In this case, we can reinforce the estimation on $\delta$ above, obtaining for all $x_0 \in \mathbb{R}^n$ that 
    \begin{align*}
        \delta \le&
        \frac{\lambda_{\mu,p}}{2\lambda_{\nu,p}} 
        \inner{\nu - \mu, h_{E_{\mu,p}}- h_{E_{\nu,p} + x_0}}
        + 
        \frac{\lambda_{\mu,p} - \lambda_{\nu,p}}{2\lambda_{\nu,p}} 
        \inner{\mu,  h_{E_{\mu,p}}- h_{E_{\nu,p} + x_0}} \\
        &+ 
        \frac{\lambda_{\mu,p}}{2\lambda_{\nu,p}} 
        \int_{\mathbb{S}^{n-1}}
        \inner{\theta, x_0}\dd(\nu - \mu)(\theta)
        +
        \frac{\lambda_{\mu,p} - \lambda_{\nu,p}}{2\lambda_{\nu,p}} 
        \int_{\mathbb{S}^{n-1}}
        \inner{\theta,x_0} \dd \mu (\theta)\\ 
        &\le 
        C_{n,\vartheta,p} \dc(\mu,\nu) \dH(E_{\mu,p}, x_0 + E_{\nu,p}),
   \end{align*}
   where we have used the Lipschitz behavior of $\mu \mapsto \lambda_{\mu,p} \eqdef n |E_{\mu,p}|^{\frac{n-p}{n}}$, as shown in~\ref{lemma.variational_problem_Lp}.

   Therefore, as for $p=1$, we summarize the informaiton comming from the quantitative $L^p$ Brunn-Minkowski inequality, for both the Fraenkel asymmetry and Hausdorff distance versions in the following
      \begin{align}
        \label{eq.estimate_deltaBM}
        \delta_{\BM,p}(E_{\mu,p},E_{\nu,p}) 
        &\le 
        C_1 \dc(\mu,\nu) \inf_{x_0 \in \mathbb{R}^n} \dH(E_{\mu,p}, x_0 + E_{\nu,p})\\
        \label{eq.BM_AMP}
        \alpha(E_{\mu,p},E_{\nu,p})^2
        &\le 
        C_2 \delta_{\BM,p}(E_{\mu,p},E_{\nu,p}) \\
        \label{eq.BM_Diskant}
        \inf_{x_0 \in \mathbb{R}^n} 
        \dH(E_{\mu,p}, x_0 + \lambda E_{\nu,p})^n
        &\le 
        C_3 \delta_{\BM,p}(E_{\mu,p},E_{\nu,p}),
   \end{align}  
   where $\lambda = \left(\frac{|E_{\mu,p}|}{|E_{\nu,p}|}\right)^{\frac{n-p}{n}}$. 

   With the same arguments as for $p=1$, we can remove the dilation $\lambda$ in~\eqref{eq.BM_Diskant} at the cost of a bigger stability constant, and we obtain the estimates for $\mu,\nu$ centered at the origin:
    \begin{align*}
        \inf_{x_0 \in \mathbb{R}^n} 
        \dH(E_{\mu,p}, x_0 + E_{\nu,p})
        \le 
        C_{n,p,\vartheta} {\dc(\mu,\nu)}^{\frac{1}{n-1}}, 
        \quad 
        \alpha(E_{\mu,p},E_{\nu,p})^2
        \le 
        C_{n,p,\vartheta} \dc(\mu,\nu)^{1 + \frac{1}{n-1}}.
    \end{align*}
    This finishes the proof of Theorem~\ref{thm.quantitative_stability_p}.
\end{proof}

% \section{Polyhedral approximation of convex bodies}\label{sec.polyhedral_approximation}

% Taking inspiration from~\cite{abdallah2015reconstruction}, we can use the stability results obtained for the Minkowski problem to obtain approximation results for convex bodies by means of polyhedral sets, as well as for the reconstruction of convex bodies by means of random normal measurements. 

% The question of optimal approximation of convex bodies by means of polyhedral sets is a classical one in convex geometry, see for instance~\cite{bronstein2008approximation}, or~\cite{mehamdi2025duality} for a more recent approach. Letting $\mathcal{P}_N$ denote the family of polyhedral sets with at most $N$ faces, and a suitable distance $\dist$ between convex bodies, given a convex body $K$ we seek to estimate 
% \[
%     \min_{P \in \mathcal{P}_N} \dist(K,P), 
% \]
% or at least to construct a polyhedral set $P_N$ with at most $N$ that is asymptotically optimal as $N \to \infty$. Classical results on this litterature give a rate of approximation for $\mathscr{C}^2$ bodies $K$ in the order $N^{-\frac{2}{n-1}}$ in the Hausdorff distance, with a constant depending on the dimension and the Gaussian curvature of the $K$~\cite{bronstein2008approximation}.

\bibliographystyle{alpha}
\bibliography{main.bib}

\end{document}